\documentclass[a4paper,twoside,10pt]{article}

\usepackage[a4paper,left=3cm,right=3cm, top=3cm, bottom=3cm]{geometry}

\usepackage{amsmath}
\usepackage{amssymb}
\usepackage{amsthm}

\usepackage{mathrsfs}
\usepackage{graphicx}
\usepackage{epstopdf}
\usepackage{amsmath}
\usepackage{amsthm}
\usepackage{amssymb}
\usepackage{stmaryrd}
\usepackage{float}
\usepackage{bigints}
\usepackage{cite}
\usepackage{color}
\usepackage[abs]{overpic}
\usepackage[font=footnotesize,labelfont=bf]{caption}
\usepackage{cases}
\usepackage{tikz}
\usepackage{algorithmic}
\usepackage{rotating}
\usepackage{blkarray}
\usetikzlibrary{matrix,calc,arrows}
\usepackage[author={Lorenzo}]{pdfcomment}
\usepackage{soul,xcolor}
\usepackage{enumitem}  
\usepackage{verbatim}
\usepackage{graphicx}
\usepackage{subcaption}
\usepackage{mwe}
\usepackage{algorithm}

\usepackage{bm,bbm}

\theoremstyle{plain}
\newtheorem{thm}{Theorem}[section]

\newtheorem{lem}[thm]{Lemma}
\newtheorem{rem}[thm]{Remark}
\newtheorem{prop}[thm]{Proposition}
\theoremstyle{definition}

\newcommand{\Th}{\Omega_{\hh}}

\newcommand{\taun}{\Omega_{\hh}}
\renewcommand{\P}{K}
\newcommand{\E}{K}
\newcommand{\F}{F}

\newcommand{\e}{e}
\newcommand{\V}{\textsf{v}}
\newcommand{\hh}{h}

\newcommand{\mP}{\vert\P\vert}
\newcommand{\mF}{\vert\F\vert}

\newcommand{\hP}{\hh_\P}
\newcommand{\hE}{\hh_\E}
\newcommand{\hF}{\hh_\F}
\newcommand{\he}{\hh_\e}

\newcommand{\restrict}[2]{{#1}_{|#2}}

\newcommand{\DIM}{d}

\newcommand{\RHS}{\textrm{RHS}}
\newcommand{\LHS}{\textrm{LHS}}
\newcommand{\TERM}[2]{\textrm{#1}_{#2}}
\newcommand{\ASSUM}[2]{\mathbf{({#1}_{#2}})}

\newcommand{\EOD}{\end{document}}

\DeclareMathOperator{\curl} {\mathbf{curl}}
\DeclareMathOperator{\curlF}{\mathbf{curl}_\F}

\let\div\relax
\DeclareMathOperator{\div}{div}
\DeclareMathOperator{\divF}{div_\F}
\DeclareMathOperator{\rotF}{rot_\F}

\DeclareMathOperator{\DeltaF}{\Delta_{\F}}

\DeclareMathOperator{\curlbold}{\mathbf{curl}}

\newcommand{\bv}{\mathbf{b}}
\newcommand{\cv}{\mathbf{c}}

\newcommand{\ev}{\mathbf{e}}
\newcommand{\fv}{\mathbf{f}}

\newcommand{\jv}{\mathbf{j}}

\newcommand{\nv}{\mathbf{n}}

\newcommand{\qv}{\mathbf{q}}

\newcommand{\tv}{\mathbf{t}}
\newcommand{\uv}{\mathbf{u}}
\newcommand{\vv}{\mathbf{v}}
\newcommand{\wv}{\mathbf{w}}
\newcommand{\xv}{\mathbf{x}}
\newcommand{\yv}{\mathbf{y}}

\newcommand{\Av}{\mathbf{A}}
\newcommand{\Bv}{\mathbf{B}}
\newcommand{\Cv}{\mathbf{C}}

\newcommand{\Ev}{\mathbf{E}}
\newcommand{\Fv}{\mathbf{F}}

\newcommand{\Vv}{\mathbf{V}}

\newcommand{\wvV}{\wv_{\V}}
\newcommand{\xvV}{\xv_{\V}}

\newcommand{\bve}{\bv_\e}
\newcommand{\bvF}{\bv_\F}
\newcommand{\bvE}{\bv_\E}
\newcommand{\bvP}{\bv_\P}

\newcommand{\xvE}{\xv_\E}
\newcommand{\xvF}{\xv_\F}
\newcommand{\xvP}{\xv_\P}

\newcommand{\zerov}{\mathbf{0}}

\newcommand{\uvt}{\uv_t}
\newcommand{\Bvt}{\Bv_t}

\newcommand{\nvF}{\nv_{\F}}
\newcommand{\tauvF}{{\bm\tau}_{\F}}

\newcommand{\evh}{\mathbf{e}_{\hh}}

\newcommand{\jvh}{\mathbf{j}_{\hh}}
\newcommand{\jvht}{\mathbf{j}_{\hh,t}}

\newcommand{\uvh}{\mathbf{u}_{\hh}}
\newcommand{\vvh}{\mathbf{v}_{\hh}}
\newcommand{\wvh}{\mathbf{w}_{\hh}}

\newcommand{\Bvh}{\mathbf{B}_{\hh}}
\newcommand{\Cvh}{\mathbf{C}_{\hh}}

\newcommand{\Evh}{\mathbf{E}_{\hh}}
\newcommand{\Fvh}{\mathbf{F}_{\hh}}

\newcommand{\Vvh}{\mathbf{V}_{\hh}}
\newcommand{\Wvh}{\mathbf{W}_{\hh}}

\newcommand{\Zvh}{\mathbf{Z}_{\hh}}

\newcommand{\evhu}{\evh^{\uv}}

\newcommand{\evhE}{\evh^{\Ev}}
\newcommand{\evhB}{\evh^{\Bv}}
\newcommand{\evhJ}{\evh^{\jv}}
\newcommand{\evhBt}{\ev^{\Bv}_{\hh,\ts}}
\newcommand{\evhut}{\ev^{\uv}_{\hh,\ts}}

\newcommand{\uvht}{\uv_{\hh,t}}
\newcommand{\Bvht}{\Bv_{\hh,t}}

\newcommand{\INTP}{I}
\newcommand{\INTPINTP}{\INTP\!\INTP}

\newcommand{\jvI}{\mathbf{j}_{\INTP}}

\newcommand{\uvI}{\mathbf{u}_{\INTP}}

\newcommand{\BvI}{\mathbf{B}_{\INTP}}

\newcommand{\EvI}{\EvII}

\newcommand{\BvII}{\mathbf{B}_{\INTP\!\INTP}}
\newcommand{\EvII}{\mathbf{E}_{\INTP\!\INTP}}
\newcommand{\BvIIt}{\mathbf{B}_{\INTP\!\INTP,t}}
\newcommand{\uvIt}{\uv_{\INTP,t}}
\newcommand{\uvtI}{\uv_{t,\INTP}}

\newcommand{\Nbb}{\mathbb{N}}

\newcommand{\Pbb}{\mathbb{P}}

\newcommand{\Rbb}{\mathbb{R}}

\newcommand{\as}{a}
\newcommand{\bs}{b}
\newcommand{\cs}{c}

\newcommand{\ms}{m}

\newcommand{\ps}{p}
\newcommand{\qs}{q}

\renewcommand{\ss}{s}
\newcommand{\ts}{t}
\newcommand{\us}{u}
\newcommand{\vs}{v}
\newcommand{\ws}{w}
\newcommand{\xs}{x}
\newcommand{\ys}{y}
\newcommand{\zs}{z}

\newcommand{\As}{A}
\newcommand{\Bs}{B}
\newcommand{\Cs}{C}
\newcommand{\Ds}{D}
\newcommand{\Es}{E}

\newcommand{\Hs}{H}
\newcommand{\Is}{I}

\newcommand{\Ss}{S}
\newcommand{\Ts}{T}

\newcommand{\Vs}{V}

\newcommand{\Xs}{X}

\newcommand{\asP}{\as^{\P}}
\newcommand{\bsP}{\bs^{\P}}
\newcommand{\csP}{\cs^{\P}}

\newcommand{\xsP}{\xs_{\P}}
\newcommand{\ysP}{\ys_{\P}}
\newcommand{\zsP}{\zs_{\P}}

\newcommand{\xsF}{\xs_{\F}}
\newcommand{\ysF}{\ys_{\F}}
\newcommand{\zsF}{\zs_{\F}}

\newcommand{\xsE}{\xs_{\e}}
\newcommand{\ysE}{\ys_{\e}}
\newcommand{\zsE}{\zs_{\e}}

\newcommand{\wsF}{\ws_{\F}}

\newcommand{\ash}{a_{\hh}}

\newcommand{\msh}{m_{\hh}}

\newcommand{\psh}{p_{\hh}}
\newcommand{\qsh}{q_{\hh}}
\newcommand{\rsh}{r_{\hh}}

\newcommand{\vsh}{v_{\hh}}

\newcommand{\Qsh}{Q_{\hh}}

\newcommand{\Vsh}{V_{\hh}}

\newcommand{\ashP}{\ash^{\P}}
\newcommand{\mshP}{\msh^{\P}}

\newcommand{\calC}{\mathcal{C}}

\newcommand{\calE}{\mathcal{E}}
\newcommand{\calF}{\mathcal{F}}
\newcommand{\calG}{\mathcal{G}}

\newcommand{\calI}{\mathcal{I}}

\newcommand{\calT}{\mathcal{T}}

\newcommand{\calV}{\mathcal{V}}

\newcommand{\calFh}{\calF_{\hh}}

\newcommand{\HS}[1]{H^{#1}}
\newcommand{\LS}[1]{L^{#1}}
\newcommand{\CS}[1]{C^{#1}}
\newcommand{\WS}[1]{W^{#1}}
\newcommand{\PS}[1]{\mathbbm{P}_{#1}}

\newcommand{\HSzr}[1]{H^{#1}_{0}}
\newcommand{\LSzr}[1]{L^{#1}_{0}}

\DeclareMathOperator{\NODE}{node}
\DeclareMathOperator{\EDGE}{edge}
\DeclareMathOperator{\FACE}{face}
\DeclareMathOperator{\CELL}{cell}

\newcommand{\VshNode}  {\Vsh^{\NODE}}

\newcommand{\setNode} {\calV_{\hh}}
\newcommand{\setNodeP}{\calV^{\P}_{\hh}}
\newcommand{\setNodeF}{\calV^{\P}_{\hh}}

\newcommand{\setEdge} {\calE_{\hh}}
\newcommand{\setEdgeP}{\calE^{\P}_{\hh}}
\newcommand{\setEdgeF}{\calE^{\F}_{\hh}}

\newcommand{\VvhEdge} {\Vvh^{\EDGE}}

\newcommand{\setFace} {\calF_{\hh}}
\newcommand{\setFaceP}{\calF^{\P}_{\hh}}

\newcommand{\VvhFace}{\Vvh^{\FACE}}

\newcommand{\VvhtFace}{\widetilde{\Vv}_{\hh}^{\FACE}}

\newcommand{\setCell}{\calC_{\hh}}
\newcommand{\VsCell} {\Vs^{\CELL}}
\newcommand{\VshCell}{\VsCell_{\hh}}

\newcommand{\scalFace}[2]{\big[#1,#2\big]_{\FACE}}
\newcommand{\scalEdge}[2]{\big[#1,#2\big]_{\EDGE}}

\newcommand{\scalEdgeP}[2]{\big[#1,#2\big]_{\EDGE,\P}}

\newcommand{\scal}[2]{\big(#1,#2\big)}

\newcommand{\chih}{{\bm\chi}_{\hh}}
\newcommand{\chihP}{{\bm\chi}_{\hh,\P}}

\newcommand{\Pin} [1]{\Pi^{\nabla}_{#1}}
\newcommand{\PinP}[1]{\Pi^{\nabla,\P}_{#1}}
\newcommand{\Pinomega}[1]{\Pi^{\nabla,\omega}_{#1}}

\newcommand{\Piz}[1]{\Pi^0_{#1}}
\newcommand{\PizP}[1]{\Pi^{0,\P}_{#1}}
\newcommand{\Pizomega}[1]{\Pi^{0,\omega}_{#1}}

\newcommand{\Pinv}[1]{\Pi^{\nabla}_{#1}}
\newcommand{\PinvP}[1]{\Pi^{\nabla,\P}_{#1}}
\newcommand{\Pinvomega}[1]{\Pi^{\nabla,\omega}_{#1}}

\newcommand{\Pizv}[1]{\boldsymbol{\Pi}^0_{#1}}
\newcommand{\PizvP}[1]{\boldsymbol{\Pi}^{0,\P}_{#1}}

\newcommand{\PizboldhP}{\boldsymbol \Pi^{0,\P}_{\hh}}

\newcommand{\PinablaF}{\Pi^{\nabla,\F}}

\newcommand{\PiboldnablaF}{\boldsymbol\Pi^{\nabla,\F}}
\newcommand{\PiboldnablaE}{\boldsymbol\Pi^{\nabla,\E}_1}

\newcommand{\taubold}{\boldsymbol \tau}

\DeclareMathOperator{\dof}{dof}

\DeclareMathOperator{\dofiP}{\dof^{\P}_i}

\newcommand{\SPa}[1]{\Ss^{\P}_{#1}}
\newcommand{\SPm}[1]{\tilde{\Ss}^{\P}_{#1}}

\newcommand{\SP}{\Ss^{\P}}

\newcommand{\SPEdge}{\SP_{\EDGE}}
\newcommand{\SPFace}{\SP_{\FACE}}

\newcommand{\ABS} [1]{\left\vert #1 \right\vert}
\newcommand{\norm}[2]{\Vert #1 \Vert_{#2}}
\newcommand{\snorm}[2]{\vert #1 \vert_{#2}}

\newcommand{\varepsilontilde}{\widetilde\varepsilon}

\DeclareMathOperator{\inv}{inv}
\DeclareMathOperator{\appr}{appr}
\newcommand{\CI}  {\Cs_{\INTP}}
\newcommand{\CC}  {C_C}
\newcommand{\CSob}{C_S}
\newcommand{\CD}{C_D}
\newcommand{\CP}{C_P}
\newcommand{\Cinv}  {\Cs_{\inv}}
\newcommand{\Cappr}  {\Cs_{\appr}}
\newcommand{\CII}  {\Cs_{\INTPINTP}}
\newcommand{\Cpi}  {\Cs_{\pi}}
\newcommand{\Cintp}{\Cs_{\INTP}}
\newcommand{\Cnst}[1]{\Cs_{#1}}

\usepackage{color}

\newcommand{\cE}{\cs^\E}
\newcommand{\cn}{\cs_\hh}
\newcommand{\ch}{\cs_\hh}

\newcommand{\cnE}{\cn^\E}
\newcommand{\chE}{\ch^\E}

\newcommand{\cht}{\widetilde{\cs}_\hh}
\newcommand{\cst}{\widetilde{\cs}}

\newcommand{\ctilde}{\widetilde c}
\newcommand{\ctildeE}{\widetilde c^\E}
\newcommand{\ctilden}{\widetilde c_\hh}

\newcommand{\ctildenE}{\widetilde c_\hh^\E}

\DeclareMathOperator{\tr}{tr}
\DeclareMathOperator{\trcurl}{\tr_{\curl}}
\DeclareMathOperator{\trdiv}{\tr_{\div}}

\DeclareMathOperator{\IntpVh}{\calI^{\setNode}}
\DeclareMathOperator{\IntpEh}{\calI^{\setEdge}}
\DeclareMathOperator{\IntpFh}{\calI^{\setFace}}
\DeclareMathOperator{\IntpPh}{\calI^{\setCell}}

\title{The virtual element method for the 3D resistive magnetohydrodynamic model}

\author{\normalsize{Louren\c{c}o Beir\~ao~da~Veiga\thanks{Dipartimento di Matematica e Applicazioni,
    Universit\`a degli Studi di Milano-Bicocca, Italy (lourenco.beirao@unimib.it, franco.dassi@unimib.it, lorenzo.mascotto@unimib.it)},\;
Franco Dassi\footnotemark[1],\;
Gianmarco Manzini\thanks{IMATI-CNR, Pavia, Italy (marco.manzini@imati.cnr.it)},\;
Lorenzo Mascotto\footnotemark[1]
\thanks{Faculty of Mathematics, University of Vienna, 1090 Vienna, Austria (lorenzo.mascotto@univie.ac.at)}
\footnotemark[3]
\thanks{Corresponding author}}}
\normalsize

\date{}
  
\begin{document}   
\maketitle

\begin{abstract}
\noindent
We present a four-field virtual element discretization for the time-dependent resistive Magnetohydrodynamics equations
in three space dimensions, focusing on the semi-discrete formulation.
The proposed method employs general polyhedral meshes and guarantees velocity and magnetic fields that are divergence free up to machine precision.
We provide a full convergence analysis under suitable regularity assumptions,
which is validated by some numerical tests.
  

  
\end{abstract}


\raggedbottom
\setcounter{secnumdepth}{4}
\setcounter{tocdepth}{4}


\section{Introduction}

Magnetohydrodynamics (MHD) is the physical-mathematical framework that
describes the dynamics of magnetic fields in electrically conducting
fluids, e.g., plasmas and ionized
gases~\cite{Boyd-Sandserson:2003-book,Brackbill:1985}.
Indeed, the electrically charged particles moving in a plasma
generate an electromagnetic field that self-consistently interacts
with the fluid motion through the Lorentz force acting on such
particles.
The theoretical framework of MHD has widely been used to develop
predictive mathematical models, for example, in astrophysics (solar
wind and space weather~\cite[Chapter~1]{Feng:2020}, galactic
jets~\cite{Komissarov-Porth:2021}), geophysics (dynamo
theory~\cite{Krause-Radler:1980}), and nuclear fusion (design of
fusion reactors, ignition in \emph{inertial confinement
fusion}~\cite{Craxton:2015,Johnson:1984,Gaffney:2018}).
Roughly speaking, MHD resorts on a strongly nonlinear coupling of a
fluid flow submodel and an electromagnetic one.
In a general setting, as the one considered in this work, the fluid
flow submodel is based on the incompressible Navier-Stokes equations
and the electromagnetic submodel on the Maxwell equations.
A thorough description of the MHD models, their variants, derivation,
and physical and mathematical properties can be found in many
textbooks and review papers,
e.g.,~\cite{Davidson:2002-book,Moreau:1990-book,Boyd-Sandserson:2003-book,Brackbill:1985}
just to mention a few.

Despite its ubiquity in scientific applications, which is reflected by
the number of papers and books published in the last few decades, the
computer resolution of the system of the MHD equations is still a
formidable and challenging task and may pose pitfalls to the numerical
scientists.
Here, we shortly review some of the more critical points, with
particular attention on the usefulness of polyhedral meshes and
related schemes, always referring all interested readers to the rich
technical literature for specific and detailed expositions.

First, vorticial and/or shear flows may lead to large domain
deformations, i.e., severe mesh deformations, also with the occurence
of non-planar mesh faces, in the Lagrangian and Arbitrary
Lagrangian-Eulerian (ALE) numerical
frameworks~\cite{Hirt-Amsden-Cook:1974,Robinson-et-al:2008}.
Such mesh deformations are treated by remeshing the computational
domain; then, by remapping the unknowns to the new
mesh~\cite{Rieben-White-Wallin-Solberg:2007}.
An optimal remap algorithm should be accurate, cheap, and
feature-preserving, i.e., it should preserve the divergence-free
constraints, the solution positivity, and conserve mass, momentum, and
energy, cf.~\cite{Burton-Morgan-Charest-Kenamond-Fung:2018}.
However, current remap algorithms can be too expensive, prone to
significant accuracy losses, and not feature-preserving.
Therefore, we infer that a good approximation method should be
accurate and stable even in presence of severe mesh deformations in
order to reduce the number of remeshing steps in a numerical
simulation.

Next, low-order accurate algorithms can be affected by an excessive
numerical diffusion and dispersion.
For example, numerical diffusion can smear physical vorticity and
shocks at an unacceptable level and even force the magnetic
reconnection in ideal MHD models, which is clearly an artificial and
physically meaningless effect.
Numerical diffusion and dispersion can be reduced by increasing the
order of the approximation and by using carefully selected polygonal
or polyhedral meshes.
In~\cite{Ding-Yang:2017,Joaquim-Scheer:2016}, the finite difference time domain (FDTD) method~\cite{Yee:1966} is applied to a grid of hexagonal
prisms and yields much less numerical dispersion and anisotropy than
on regular hexahedral grids where such method is normally considered.

The final major point that we want to consider is related to the
divergence-free nature of both the fluid velocity and the magnetic flux fields.
If such constraints are not accurately satisfied, possibly at the
machine precision level, unreliable or even unphysical solutions may
result from a numerical simulation.
The consequence of the violation of the divergence-free constraint for
the magnetic flux field has widely been investigated in the literature.
It was seen that the numerical simulations are prone to significant
errors, see, e.g.,
\cite{Brackbill:1985,Brackbill-Barnes:1980,Dai-Woodward:1998,Toth:2000},
as fictitious forces and an unphysical behavior may appear,
cf.~\cite{Dai-Woodward:1998}.
For this reason, in the last two decades, a great effort has been
devoted to the development of divergence-free numerical
approximations.
Just to mention a few possible approaches, the divergence-free
constraint of the magnetic flux field can explicitly be enforced by
introducing a Lagrange multiplier in the set of the unknowns,
cf.~\cite{Dedner-Kemm-Kroner-Munz-Schnitzer-Wesenberg:2002}; by using
a special flux limiter in the formulation of the numerical scheme,
cf.~\cite{Kuzmin-Klyushnev:2020}; by minimizing a special energy
functional in a least squares finite element formulation,
cf.~\cite{Jiang:1998}.
However, such fixing strategies can result in costly inefficient
schemes, they can be in conflict with each other or not suitable to
polyhedral meshes.
A major breakthrough was provided by realizing that ``classical''
numerical discretizations fail to reproduce the divergence-free nature
of these fields if the discrete version of the divergence applied to
the discrete version of the rotational operator does not annihilate at
the zero machine precision.
In fact, a small but not zero remainder can significantly accumulate
over the many cycles of a long-time calculation, thus breaking the
constraint.
This understanding has led to the design of numerical approximation
techniques in the framework of compatible/mimetic methods, as for
example
in~\cite{Rieben-White-Wallin-Solberg:2007,Robinson-et-al:2008,Lipnikov-Reynolds-Nelson:2013}.
A major drawback of compatible/mimetic discretizations
is that they are only low-order accurate and can be too dispersive
and diffusive.
The discontinuous Galerkin (DG) method offers high-order accurate
discretizations suitable to general polyhedral meshes that can reduce
the overall numerical diffusion and dispersion.
Nonetheless, DG methods are not compatible and may require a costly
divergence-cleaning procedure based on solving an additional global
equation, thus resulting in very expensive calculations.
Finally, we note that the FEM literature for MHD models on tetrahedral
and hexahedral meshes is very broad.
Here, we limit ourselves to recalling the divergence-free schemes
in~\cite{Hu-Ma-Xu:2017,Greif-Li-Schoetzau-Wei:2010,Hiptmair-Li-Mao_Zheng:2018},
see the related work~\cite{Schoetzau:2004} as well, and the partial
convergence analysis results from~\cite{He:2015, Prohl:2008}

\medskip
In this work, we start exploring and presenting the design of
numerical approximations of the MHD equations based on the virtual element method (VEM).
Our goal is to address at least part of the issues discussed so far,
but in a framework that makes it possible to address all of them in
the future developments of this research project.
The VEM, initially proposed
in~\cite{BeiraodaVeiga-Brezzi-Cangiani-Manzini-Marini-Russo:2013} for
the numerical approximation of elliptic problems, is a Galerkin
projection method such as the finite element method (FEM).
The major difference between the VEM and the FEM is that in the formulation
and practical implementation of the VEM we do not need an explicit
knowledge of the basis functions that generate the finite element
approximation spaces.
In fact, all the bilinear forms and linear functionals of the discrete
variational formulation are built using suitable polynomial
projections that are always computable from a careful choice of the
degrees of freedom.
For these reasons, such approximation spaces and the method itself are
dubbed as ``virtual''.
Local polynomial consistency and an additional stability term provide
the well-posedness of the final discretization.
Since the formulation of the method does not need a closed form of the
basis functions, the resulting computational framework is extremely
powerful and offers important advantages with respect to the FEM.
We can build approximation spaces that are in principle of any order
of accuracy and global regularity, suitable to general polyhedral meshes,
even including nonconvex or nonconforming elements in 2D and 3D, and,
importantly, satifying other additional properties such as being part
of a discrete de Rham chain.
This last property has the major consequence that both the numerical
approximation of the fluid velocity and the magnetic flux field are
intrinsically divergence free, a fact that we can verify numerically
up to the machine precision.

The discrete divergence-free property in the VEM has been firstly
analyzed
in~\cite{BeiraodaVeiga-Lovadina-Vacca:2017,BeiraodaVeiga-Lovadina-Vacca:2018}
(see also~\cite{Antonietti-BeiraodaVeiga-Mora-Verani:2014}) in the
framework of the Stokes and Navier-Stokes problems.
Later, it has been extended to other fluid flow models, a short list
of representatives
being~\cite{Cangiani-Gyrya-Manzini:2016,%
  Chen-Li-Liu:2017,%
  Chernov-Marcati-Mascotto:2021,%
  Chen-Liu:2019, Caceres-Gatica:2017,%
  Caceres-Gatica:2018,%
  Gatica-Munar-Sequeira:2018,%
  Chen-He-Wang-Wang:2019,%
  Vacca:2018,%
  Caceres-Gatica-Sequeira:2017,%
  BeiraodaVeiga-Mora-Vacca:2019,%
  BeiraodaVeiga-Dassi-Vacca:2020,%
  Irisarri-Gauke:2019,%
  Dassi-Vacca:2020}.
On the other hand, De-Rahm complexes in the virtual element setting
have been introduced in~\cite{BeiraodaVeiga-Brezzi-Marini-Russo:2016}
and then improved (and applied to the analysis of magnetostatic
problems) in
\cite{BeiraodaVeiga-Brezzi-Dassi-Marini-Russo:Magneto2017,%
  BeiraodaVeiga-Brezzi-Dassi-Marini-Russo:2018,%
  BeiraodaVeiga-Brezzi-Dassi-Marini-Russo:2018SIAM}.
The Maxwell's equations were the target
of~\cite{BeiraodaVeiga-Dassi-Manzini-Mascotto:2021},
see also~\cite{Cao-Chen-Guo:2021, Cao-Chen-Guo-Lin:2021}, and their analysis hinges upon
technical tools from~\cite{BeiraodaVeiga-Mascotto:2021}.  Other
interesting results along these directions are contained
in~\cite{Chen-Huang:2020} that paves the way for a VEM discretization
of the Bianchi-Einstein equations.

In the present contribution we present a low-order four-field VEM
formulation of the time-dependent MHD equations in three dimensions,
in the spirit of \cite{Hu-Ma-Xu:2017}.
Compared to~\cite{Hu-Ma-Xu:2017}, we further prove convergence of the semi-discrete scheme
and employ a convenient divergence-free discretization of the velocity field
(whereas the Taylor-Hood element is used in~\cite{Hu-Ma-Xu:2017}).
In order to discretize the unknown fields (pressure, and velocity,
magnetic, and electric fields), we exploit the compatible VEM discrete
spaces outlined above.
Since the novelty of the method is in the space discretization
approach, we focus on the semi-discrete version of the scheme, and
leave the development of different time advancing techniques, together
with the associated linear and nonlinear solvers, to future
publications.
The ensuing VEM scheme employs general polyhedral meshes and,
differently from the majority of the FEM schemes available in the literature,
guarantees that both the velocity and magnetic fields are divergence
free (up to machine precision).
An important achievement of the present contribution is that, under
suitable regularity conditions on the exact solution, we are able to
prove the linear convergence of the scheme in the ``natural'' norms of
the problem.
To the best of our knowledge, there are few convergence analyses of FEM schemes for the time-dependent MHD model in the literature,
and none for the 4-field formulation here presented.
For instance, in~\cite{Hu-Qiu-Shi:2020}, a lowest order FEM (with discrete velocity that is not divergence free) has been designed and analyzed for the \emph{stationary} MHD system.
Thus, a major contribution of this work is that we also provide the
details for the convergence analysis for the fully nonlinear model;
such an analysis can be simplified to the FE setting, thus leading to
possibly useful results also for FEM.
Clearly, our analysis has the drawback of requiring a sufficiently
regular exact solution; thus, our theoretical results should be
intended as a guarantee that, at least in good conditions, the method
delivers an accurate solution.
In the final part of the article, we present some numerical tests that
show the good performance of the proposed method.

The paper is organized as follows. 
After discussing some basic notation and definitions in Section~\ref{section:notation},
we review the continuous problem in
Section~\ref{section:continuous-problem}.
In Section~\ref{section:VEM:Navier-Stokes}, we introduce the
velocity/pressure discrete fields and some theoretical results
essential to the approximation of the convection term.
In Section~\ref{section:VEM:Maxwell}, after presenting the discrete
spaces adopted for the electromagnetic fields and some related
property, we state the discrete method.
In Section~\ref{section:VEM:analysis}, we show and prove the
theoretical convergence results, which we validate in
Section~\ref{section:numerical:experiments} with some numerical tests.
Finally, in Section~\ref{section:conclusions}, we draw some
conclusions and outline future developments.

bigskip

\section{Notation} \label{section:notation}

\subsection{Mesh notation and functional spaces}
\label{subsubsection:meshes}

Let $\calT=\{\Th\}_{\hh}$ be a family of mesh decompositions $\Th$ of
the computational domain $\Omega$ uniquely identified by the value of
the mesh size parameter $\hh$.
Every mesh $\Th$ is a finite collection of polytopal elements $\P$
forming a finite covering of $\Omega$, i.e.,
$\overline{\Omega}=\bigcup_{\P\in\Th}\P$, planar faces $\setFace$,
straight edges $\setEdge$, and vertices $\setNode$.
The mesh elements are nonoverlapping in the sense that the
intersection in $\Rbb^{3}$ of any pair of them, e.g., $\P$ and
$\P^{\prime}$, has Lebesgue measure (volume) equal to zero, i.e.,
$\ABS{\P\cap\P^{\prime}}=0$ if $\P\neq\P^{\prime}$.
Accordingly, the intersection of the elemental boundaries of $\P$ and
$\P^{\prime}$ is either the empty set, or a set of common vertices,
edges, or faces.

\medskip
For every element~$\P$, we denote its volume, center, and diameter by~$\mP$, $\bvP=(\xsP,\ysP,\zsP)^T$, and~$\hP=\max_{\xv,\yv\in\P}\ABS{\xv-\yv}$,
and its set of faces, edges, and vertices by~$\setFaceP$, $\setEdgeP$, and~$\setNodeP$, respectively.
As usual, the maximum of the diameters~$\hP$ for~$\P\in\Th$ is the mesh size, i.e., $\hh=\max_{\P\in\Th}\hP$.
The boundary of every element $\P$ is denoted by $\partial\P$ and formed by a finite set of nonintersecting planar faces
$\F\in\setFaceP$ so that $\partial\P=\bigcup_{\F\in\setFaceP}\F$.
For every face~$\F$, we denote its area, center, and diameter
by~$\mF$, $\bvF=(\xsF,\ysF,\zsF)^T$, and~$\hF=\max_{\xv,\yv\in\F}\ABS{\xv-\yv}$,
and its set of edges and vertices by~$\setEdgeF$ and~$\setNodeF$, respectively.
The boundary of every face $\F$ is the planar polygon $\partial\F$,
which is formed by a finite number of nonintersecting straigh segments
$\e\in\setEdgeF$ so that $\partial\F=\bigcup_{\e\in\setEdgeF}\e$.
We denote the length of edge $\e$ by $\he$ and its center, which we
shall also refer to as the edge mid-point, by $\bve=(\xsE,\ysE,\zsE)$.

\medskip
We assume that all the meshes~$\Th$ of a given sequence$\{\Th\}_{\hh}$ satisfy these conditions for~$\hh\to0$: there exists a
real constant factor~$\gamma \in (0,1)$ that is independent of $\Th$
and $\hE\in\Th$ such that
\begin{description}
\smallskip
\item[$\ASSUM{M}{1}$] \textbf{shape-regularity}: all the
  elements~$\P\in\Th$ and faces~$\F\in\setFace$ are ``$\gamma$-shape''
  regular~\cite{Ciarlet:1978};
  \smallskip
\item[$\ASSUM{M}{2}$] \textbf{uniform scaling}: $\gamma\hP\leq\hF$ for
  every face~$\F\in\setFaceP$ of every element~$\P\in\Th$, and,
  analogously, $\gamma\hF\leq\he$ for every edge~$\e\in\setEdgeF$ of
  every~$\F\in\setFaceP$.
\end{description}
Throughout the paper, we shall refer to $\ASSUM{M}{1}$-$\ASSUM{M}{2}$
as the \emph{mesh regularity assumptions}.
Such assumptions could be weakened~\cite{Brenner-Guan-Sung:2017,Brenner-Sung:2018,BeiraoDaVeiga-Lovadina-Russo:2017,Cao-Chen:2018};
however, for the sake of presentation, we stick to a simpler setting.

\medskip
On every face $\F\in\setFace$, we define a local coordinate system
$(\xi_1,\xi_2)$ and denote the differentiation along $\xi_1$ and
$\xi_2$ as $\partial_{\xi_1}$ and $\partial_{\xi_2}$, respectively.
The corresponding second order derivatives are denoted by~$\partial^2_{\xi_1 \xi_1}$ and $\partial^2_{\xi_2 \xi_2}$.
Then, we consider the two-dimensional vector-valued
field~$\vv=(\vs_1,\vs_2):\F\subseteq\Rbb^2\rightarrow\Rbb^2$ and the
scalar field~$\vs:\F\subseteq\Rbb^2\rightarrow\Rbb$, and let the
face-based differential operators $\divF$, $\rotF$, $\curlF$ and
$\DeltaF$ applied to $\vv(\xi_1,\xi_2)$ and $\vs(\xi_1,\xi_2)$ be
defined as
\begin{subequations}
  \label{eq:2D:diff:ops}
  \begin{align}
    \divF\vv   &:= \partial_{\xi_1}\vs_1 + \partial_{\xi_2}\vs_2,
    \qquad
    \rotF\vv   := \partial_{\xi_2}\vs_1 - \partial_{\xi_1}\vs_2,
    \label{eq:2D:diff:ops:A}\\[0.5em]
    \curlF\vs  &:= \big(\partial_{\xi_2}\vs, -\partial_{\xi_1}\vs\big)^T,
    \qquad
    \DeltaF\vs := \partial^2_{\xi_1\xi_1}\vs + \partial^2_{\xi_2\xi_2}\vs.
    \label{eq:2D:diff:ops:B}
  \end{align}
\end{subequations}
We denote the partial derivatives along the directions $x$, $y$,
and~$z$ by~$\partial_x$, $\partial_y$, and~$\partial_z$, respectively,
and the corresponding second derivatives by~$\partial^2_{xx}$,
$\partial^2_{yy}$, and~$\partial^2_{zz}$.
We define the Laplace, divergence, and curl operators of the
three-dimensional field~$\vv=(\vs_1,\vs_2,\vs_3)^T$ as follows:
\begin{align*}
  &
  \Delta\vv := \partial^2_{xx}\vs_1 + \partial^2_{yy}\vs_2 + \partial^2_{zz}\vs_3,\qquad
  \div  \vv := \partial_x\vs_1   + \partial_y\vs_2    + \partial_z\vs_3,\\
  &
  \curl \vv := \big(
  \partial_y\vs_3 - \partial_z\vs_2,
  \partial_z\vs_1 - \partial_x\vs_3,
  \partial_x\vs_2 - \partial_y\vs_1
  \big)^T.
\end{align*}

\subsection{Functional spaces}
\label{subsection:functional-spaces}

Consider the polygonal/polyhedral domain $\Ds\subset\Rbb^{\DIM}$, $\DIM=2,3$, with (Lipschitz) boundary~$\partial\Ds$.
Throughout the paper, $\Ds$ can be either a mesh face~$\F$, a mesh element~$\P$, or the whole domain~$\Omega$.

\paragraph*{Sobolev spaces}
According to~\cite{Adams-Fournier:2003}, $\LS{2}(\Ds)$ denotes the
Lebesgue space of real-valued square integrable functions defined on
$\Ds$;
$\LSzr{2}(\Ds)$ is the subspace of the functions in $\LS{2}(\Ds)$ having zero average on~$\Ds$;
for any integer number~$\ss>0$, $\HS{s}(\Ds)$ is the Sobolev space~$s$ of the real-valued functions in $\LS{2}(\Ds)$ with all weak partial derivatives of order up to~$\ss$ in~$\LS{2}(\Ds)$;
$\big[\LS{2}(\Ds)\big]^{\DIM}$, $\big[\LSzr{2}(\Ds)\big]^{\DIM}$, and
$\big[\HS{s}(\Ds)\big]^{\DIM}$ are the vector version of these spaces.
The Sobolev spaces of noninteger orders are constructed by using the
interpolation theory and the Sobolev spaces of negative orders by
duality~\cite{Adams-Fournier:2003}.

We denote the Sobolev space on~$\partial\Ds$ by~$\HS{\ss}(\partial\Ds)$ and recall that the functions
in~$\HS{s}(\Ds)$ admit a trace on~$\partial\Ds$ when~$s>1/2$.
Since~$\Ds$ can be either a polygonal or a polyhedral domain, the
bound~$s<3/2$ must also be valid.
We denote the inner product in $\LS{2}(\Ds)$ and $\HS{s}(\Ds)$ by
$(\vs,\ws)_{\Ds}$ and $(\vs,\ws)_{\ss,\Ds}$, respectively; we also
denote the corresponding induced norms by $\norm{\vs}{\Ds}$ and
$\norm{\vs}{\ss,\Ds}$, and the seminorm in $\HS{s}(\Ds)$ by
$\snorm{\vs}{\ss,\Ds}$.
When~~$\Ds$ is the whole computational domain, we prefer to omit the subindex~$\Omega$ and rather use the notation $(\vs,\ws)$, $\norm{\vs}{}$,
$\norm{\vs}{\ss}$, etc. instead of $(\vs,\ws)_{\Omega}$,
$\norm{\vs}{\Omega}$, $\norm{\vs}{\ss,\Omega}$, etc.

In the light of these definitions, for a given~$s>0$, we introduce the functional spaces
\begin{align}
  \HS{s}(\div,  \Ds) &:= \Big\{ \vv\in\big[\HS{s}(\Ds)\big]^3 \mid\div \vv\in\HS{s}(\Ds)     \Big\},\label{eq:Hdiv:def}\\[0.25em]
  \HS{s}(\curl, \Ds) &:= \Big\{ \vv\in\big[\HS{s}(\Ds)\big]^3 \mid\curl\vv\in\big[\HS{s}(\Ds)\big]^3\label{eq:Hcurl:def}\big\}.
\end{align}
If~$s=0$, we write~$\HS{}(\div,\Ds)$ and~$\HS{}(\curl,\Ds)$ instead of
$\HS{0}(\div,\Ds)$ and~$\HS{0}(\curl,\Ds)$.
Let~$\nv_{\Ds}$ be the unit vector orthogonal to the
boundary~$\partial\Ds$ and pointing out of~$\Ds$.
According, e.g., to~\cite[Section~$3.5$]{Monk:2003}, for the spaces~\eqref{eq:Hdiv:def} and~\eqref{eq:Hcurl:def}, we can define the trace operators
\begin{align*}
  \trdiv: \HS{}(\div, \Ds) \to  \HS{-\frac12}(\partial\Ds) , \qquad
  \trcurl: \HS{}(\curl,\Ds) \to \big[\HS{-\frac12}(\partial\Ds)\big]^{\DIM},
  \qquad \DIM =2,3,
\end{align*}
which are such that
\begin{align*}
  \trdiv (\vv):= \nv_{\Ds} \cdot \vv, \qquad
  \trcurl(\vv):= \nv_{\Ds} \times\vv,
\end{align*}
for all sufficiently smooth vector-valued field~$\vv$.
These trace operators allow us to define the subspaces of
$\HS{s}(\div,\Ds)$ and $\HS{s}(\curl,\Ds)$
\begin{align*}
  \HSzr{s}(\div, \Ds) &:= \big\{ \vv\in\HS{s}(\div, \Ds) \mid \trdiv  \vv = 0 \big\},\\[0.25em]
  \HSzr{s}(\curl,\Ds) &:= \big\{ \vv\in\HS{s}(\curl,\Ds) \mid \trcurl \vv = \mathbf{0} \big\}.
\end{align*}
These subspaces incorporate the homogeneous boundary conditions in
their definition.
Finally, we define the norms for div and curl spaces:
for every sufficiently smooth field~$\vv$,
\[
\norm{\vv}{\curlbold}^2:= \norm{\vv}{0}^2 + \norm{\curlbold \vv}{0}^2,
\qquad \qquad
\norm{\vv}{\div}^2:= \norm{\vv}{0}^2 + \norm{\div \vv}{0}^2.
\]

\paragraph*{Bochner spaces}
Let $\Ts>0$ be a real number and $(\Xs,\norm{\cdot}{\Xs})$ a normed space, where~$\Xs$ can be either~$\LS{2}(\Omega)$ or~$\HS{s}(\Omega)$, $\ss\geq 0$.
According to~\cite{Evans:2010}, the Bochner space $\LS{p}(0,T;X)$ is
the space of functions $\vs$ such that the sublinear functional
\begin{align*}
  \norm{\vs}{\LS{p}(0,\Ts;\Xs)}
  = \begin{cases}
    \displaystyle\left(\int_{0}^{\Ts}\norm{\vs(t)}{\Xs}^{p}\,dt\right)^{1\slash{p}} & 1\leq p<\infty,\\
    \textrm{ess sup}_{\ts\in[0,\Ts]}\norm{\vs(t)}{\Xs}                                  & p=\infty,
  \end{cases}
\end{align*}
is a \emph{finite} norm for almost every $t\in[0,\Ts]$.
According to this notation, $\CS{}(0,\Ts;\Xs)$ is the space of the
continuous functions from $[0,\Ts]$ to $\Xs$.

\paragraph*{Polynomial spaces}
We denote the space of polynomials of degree~$\ell=0,1$ defined on the element~$\P$, the face~$\F$ and the edge~$\e$ by~$\PS{\ell}(\P)$, $\PS{\ell}(\F)$, and~$\PS{\ell}(\e)$, respectively.
We set $\PS{-1}(\P)=\PS{-1}(\F)=\PS{-1}(\e)=\{0\}$.
The space $\PS{1}(\P)$ is the span of the \emph{scaled monomials} defined as:
\begin{align*}
  \ms_0(\xv) = 1,\;
  \ms_1(\xv) = \frac{\xs-\xsP}{\hP},\;
  \ms_2(\xv) = \frac{\ys-\ysP}{\hP},\;
  \ms_3(\xv) = \frac{\zs-\zsP}{\hP} \quad\forall \xv=(\xs,\ys,\zs)^T\in\P.
\end{align*}
The bases of $\PS{1}(\F)$ and $\PS{1}(\e)$ are defined in a similar way.
We let $\PS{\ell}(\Th)$ denote the space of the piecewise
discontinuous polynomials of degree $\ell=0,1$ that are globally
defined on $\Omega$ and such that $\restrict{\qs}{\P}\in\PS{1}(\P)$
for all elements $\P\in\Th$.

\paragraph*{Orthogonal projections onto polynomial spaces}
In the forthcoming discrete formulation, given any~$\omega$ either
in~$\Th$, $\setFace$, or~$\setEdge$, we shall use the polynomial
projectors
$\Pizomega{\ell}:\LS{2}(\omega)\rightarrow\PS{\ell}(\omega)$,
$\ell=0,1$, and $\Pinomega{1}:\HS{1}(\omega)\rightarrow\PS{1}(\omega)$
defined on all the mesh objects~$\omega$.
The operator~$\Pizomega{\ell}$ is the orthogonal projection onto
constant ($\ell=0$) and linear $(\ell=1)$ polynomials with respect to
the inner product in $\LS{2}(\omega)$.
The operator $\Pinomega{1}$ is the orthogonal projection onto linear
polynomials with respect to the semi-inner product
in~$\HS{1}(\omega)$; we call it the \emph{elliptic projection}.
The elliptic projection $\Pinomega{\ell}\vs$ of a function
$\vs\in\HS{1}(\omega)$ is the linear polynomial solving the
variational problem
\begin{align}
  \big(\nabla(\Pinomega{1}\vs-\vs),\nabla\qs\big)_{\omega} = 0 \quad\forall\qs\in\PS{1}(\omega)
  \quad\textrm{and}\quad
  \int_{\partial \omega} \big( \Pinomega{1}\vs-\vs \big) = 0.
  \label{eq:elliptic:projector}
\end{align}
The second condition in~\eqref{eq:elliptic:projector} is considered to
ensure the uniqueness of the projection $\PinP{1}\vs$.

\medskip
With an abuse of notation, we extend these definitions in a component-wise manner to the
multidimensional projection operators
$\Pizomega{\ell}:\big[\LS{2}(\omega)\big]^{\DIM}\to\big[\PS{\ell}(\omega)\big]^{\DIM}$
and
$\Pinvomega{1}:\big[\HS{1}(\omega)\big]^{\DIM}\to\big[\PS{1}(\omega)\big]^{\DIM}$, $\DIM=2,3$.
In particular, we shall use the orthogonal projection
$\PizP{0}\nabla\vs$ of the gradient of a function $\vs\in\HS{1}(\P)$,
which is the constant vector polynomial solving the variational problem:
\begin{align*}
  \big(\PizP{0}\nabla\vs-\nabla\vs,\qv\big)_{\P} = 0 \quad\forall\qv\in\big[\PS{0}(\P)\big]^3.
\end{align*}

\medskip
For~$\ell =0,1$, we also define the global projection operators
$\Piz{\ell}:\LS{2}(\Omega)\to\PS{\ell}(\Th)$ and
$\Piz{\ell}:\big[\LS{2}(\Omega)\big]^{\DIM}\to\big[\PS{\ell}(\Th)\big]^{\DIM}$
as the operators respectively satisfying
$\restrict{\big(\Piz{\ell}\vs\big)}{\P}=\PizP{\ell}\big(\restrict{\vs}{\P}\big)$
and
$\restrict{\big(\Piz{\ell}\vv\big)}{\P}=\PizP{\ell}\big(\restrict{\vv}{\P}\big)$
for all mesh elements~$\P\in\Th$.

\medskip
Similarly, $\Pin{1}:\HS{1}(\Omega)\to\PS{1}(\Th)$ and
$\Pinv{1}:\big[\HS{1}(\Omega)\big]^{\DIM}\to\big[\PS{1}(\Th)\big]^{\DIM}$
are the global projection operators respectively satisfying
$\restrict{\big(\Pin{1}\vs\big)}{\P}=\PinP{1}\big(\restrict{\vs}{\P}\big)$
and
$\restrict{\big(\Pinv{1}\vv\big)}{\P}=\PinvP{1}\big(\restrict{\vv}{\P}\big)$
for all mesh elements~$\P\in\Th$.

\paragraph*{Some names for constants}
In some of the forthcoming estimates, we shall occasionally write explicit constants and denote them with different symbols, depending on their meaning.
Notably, we shall use the following notation:
\begin{itemize}
    \item $\CSob$ denotes a constant depending on a Sobolev embedding;
    \item $\CD$ denotes a constant depending on the shape of a domain;
    \item $\CP$ denotes a constant appearing in a Poincar\'e inequality;
    \item $\CI$ denotes a constant depending on an interpolation estimate with respect to functions in virtual element spaces;
    \item $\Cinv$ denotes a constant depending on an inverse estimate;
    \item $\Cappr$ denotes a constant depending on a polynomial approximation estimate.
\end{itemize}
We deem that this notation might help the reader to better follow some steps in the forthcoming proofs.

Henceforth, we use the letter ``$\Cs$'' to denote a strictly
positive constant that can take a different value at any occurrence.
The constant $\Cs$ is independent of the mesh size parameter $\hh$ but
may depend on the other parameters of the differential problem and
virtual element discretization such as the domain shape, the mesh
regularity constant $\gamma$, and the coercivity and continuity
constants of the bilinear forms used in the variational formulations
that will be introduced in the next sections.

\newcommand{\REY}{Re}
\newcommand{\REM}{\REY_m}

\section{The continuous problem}
\label{section:continuous-problem}

Let $\Omega\subset\Rbb^3$ be a polyhedral domain with Lipschitz continuous boundary $\partial\Omega$.
We model the interaction of an electrically charged, incompressible
fluid having velocity~$\uv$ and pressure~$\ps$ with the
self-consistently generated electric field~$\Ev$ and magnetic flux
field~$\Bv$.
We denote the viscous Reynolds number by $\REY$, the magnetic Reynolds
number by $\REM$, and the Hartman number by~$s$.
Furthermore, we define the electric current density
\begin{equation} \label{jbold}
  \jv := \Ev + \uv\times\Bv.
\end{equation}
Henceforth, the subscript~$t$ as in~$\uvt$ and~$\Bvt$ denotes the time
derivative and~$[0,T]$ is the time integration interval for a given
final time~$\Ts>0$.

The MHD problem reads as follows: \emph{For all times~$t \in (0,T]$,
find~$\uv$, $\ps$, $\Ev$, and~$\Bv$, such that}
\begin{subequations}
  \label{eq:continuous:problem-strong}
  \begin{align}
    \label{} \uvt + (\nabla\uv)\uv - \REY^{-1}\Delta\uv -s \jv\times\Bv + \nabla\ps  &= \fv\phantom{0\zerov} \qquad\text{in }\Omega,\label{eq:continuous:problem-strong:a}\\[0.1em]
    \jv - \REM^{-1} \curl\Bv                                                         &= \zerov\phantom{0\fv} \qquad\text{in }\Omega,\label{eq:continuous:problem-strong:b}\\[0.1em]
    \Bvt + \curl\Ev                                                                 &= \zerov\phantom{0\fv} \qquad\text{in }\Omega,\label{eq:continuous:problem-strong:c}\\[0.1em]
    \div \Bv                                                                        &= 0\phantom{\zerov\fv} \qquad\text{in }\Omega,\label{eq:continuous:problem-strong:d}\\[0.1em]
    \div \uv                                                                        &= 0\phantom{\zerov\fv} \qquad\text{in }\Omega,\label{eq:continuous:problem-strong:e}
  \end{align}
\end{subequations}
where $(\nabla\uv)\uv=\sum_{i,j}(\partial_j\us_i)\us_j$.
The MHD equations~\eqref{eq:continuous:problem-strong} are completed
by the set of initial conditions for the velocity and the magnetic flux fields
\begin{equation} 
  \uv(\xv,0) = \uv_0(\xv),  \qquad
  \Bv(\xv,0) = \Bv_0(\xv), \qquad
  \forall\xv\in\Omega,
  \label{eq:initial:conditions}
\end{equation}
and the homogeneous boundary conditions
\begin{equation} 
  \uv(\xv,\ts)                  = \zerov,\qquad
  \Bv(\xv,\ts) \cdot \nv_{\Omega} = 0,     \qquad
  \Ev(\xv,\ts) \times\nv_{\Omega} = \zerov \qquad
  \forall\xv\in\partial\Omega,\quad\forall\ts\in [0,T].
  \label{eq:boundary:conditions}
\end{equation}
Equations~\eqref{eq:continuous:problem-strong:a}
and~\eqref{eq:continuous:problem-strong:e} describe the hydrodynamic
behavior of an electrically charged, incompressible fluid under the
action of an external force $\fv$ and the electromagnetic force
$\jv\times\Bv$ multiplied by the Hartmann number, which acts as a
coupling coefficient.
The electromagnetic force is
self-consistently generated by the electromagnetic fields $\Ev$ and
$\Bv$ satisfying equations~\eqref{eq:continuous:problem-strong:b},
\eqref{eq:continuous:problem-strong:c}, and
\eqref{eq:continuous:problem-strong:d}.
These last three equations describe the electromagnetic submodel in
the magneto-hydrodynamics approximation~\cite{Boyd-Sandserson:2003-book}.
The incompressibility of the fluid velocity and the solenoidal nature
of the magnetic field require that both $\uv_0$ and $\Bv_0$
in~\eqref{eq:initial:conditions} are divergence free.

\medskip
\noindent
The weak formulation of
problem~\eqref{eq:continuous:problem-strong}-\eqref{eq:boundary:conditions}
reads as follows:
\emph{For almost every $t\in[0,T]$, find
$(\uv(t),\ps(t)$, $\Ev(t),\Bv(t))\in\big[\HSzr{1}(\Omega)\big]^3\times
\LSzr{2}(\Omega)\times\HSzr{}(\curl,\Omega)\times\HSzr{}(\div,\Omega)$
such that}
\begin{subequations}
  \label{eq:continuous:problem-weak}
  \begin{align}
    (\uvt,\vv) + \REY^{-1}\as(\uv,\vv) + \bs(\vv,\ps) + \cs(\uv;\uv,\vv) - \ss(\jv\times\Bv,\vv)  &= (\fv,\vv)\phantom{0} \qquad\forall\vv\in\big[\HSzr{1}(\Omega)\big]^3,\label{eq:MHD:weak:A} \\[0.25em]
    (\jv,\Fv)  - \REM^{-1}(\Bv,\curl\Fv)                                                          &= 0\phantom{(\fv,\vv)} \qquad\forall\Fv\in\HSzr{}(\curl,\Omega),\label{eq:MHD:weak:B}        \\[0.25em]
    (\Bvt,\Cv) + (\curl\Ev,\Cv)                                                                  &= 0\phantom{(\fv,\vv)} \qquad\forall\Cv\in\HSzr{}(\div,\Omega),\label{eq:MHD:weak:C}         \\[0.25em]
    \bs(\uv,\qs)                                                                                 &= 0\phantom{(\fv,\vv)} \qquad\forall\qs\in\LSzr{}(\Omega),\label{eq:MHD:weak:D}
  \end{align}
\end{subequations}
where $\jv$ is defined in~\eqref{jbold}.
The bilinear forms $\as(\cdot,\cdot)$ and $\bs(\cdot,\cdot)$, and the
trilinear form $\cs(\cdot;\cdot,\cdot)$ are defined as
\begin{align*}
  \as:\big[\HS{1}(\Omega)\big]^3\times\big[\HS{1}(\Omega)\big]^3\to\Rbb:                                 & \quad\as(\uv,\vv)     := (\nabla\uv,\nabla\vv)   = \displaystyle  \int_{\Omega}\nabla\uv:\nabla\vv,\\[1.em]
  \bs:\big[\HS{1}(\Omega)\big]^3\times\LSzr{2}   (\Omega)       \to\Rbb:                                 & \quad\bs(\vv,\qs)     := -(\div\vv,\qs)          = \displaystyle -\int_{\Omega}\qs\div\vv,\\[1.em]
  \cs:\big[\HS{1}(\Omega)\big]^3\times\big[\HS{1}(\Omega)\big]^3\times\big[\HS{1}(\Omega)\big]^3\to\Rbb: & \quad\cs(\wv;\uv,\vv) := ( (\nabla\uv) \wv,\vv) = \displaystyle  \int_{\Omega}(\nabla\uv)\wv\cdot\vv,
\end{align*}
where $\Av:\Bv=\sum_{ij}\As_{ij}\Bs_{ij}$ and~$ (\Av)\bv\cdot\cv=\sum_{ij}\As_{ij}\bs_{j}\cs_{i}$ for the matrices
$\Av=(\As_{ij})$ and $\Bv=(\Bs_{ij})$, and the vectors $\bv=(\bs_i)$
and $\cv=(\cs_i)$.
We also consider the local forms that are defined by splitting the
above forms on the mesh elements $\P\in\Th$:
\begin{align*}
  \as(\uv,\vv)     &=\sum_{\P\in\Th}\asP(\uv,\vv)        \quad\textrm{and}\quad\asP(\uv,\vv)     = (\nabla\uv,\nabla\vv)_{\P}, \\[0.5em]
  \bs(\vv,\qs)     &=\sum_{\P\in\Th}\bsP(\uv,\qs)        \quad\textrm{and}\quad\bsP(\uv,\qs)     = -(\div\vv,\qs)_{\P},        \\[0.5em]
  \cs(\vv;\uv,\wv) &=\sum_{\P\in\Th}\csP(\vv;\uv,\wv)    \quad\textrm{and}\quad\csP(\vv;\uv,\wv) = ((\nabla\uv)\vv,\wv)_{\P}.
\end{align*}

\medskip
Several (partial) results are available from the technical literature
about the well-posedness of the MHD model in the strong and weak
formulation, e.g., problems~\eqref{eq:continuous:problem-strong}
and~\eqref{eq:continuous:problem-weak} and their variants, cf.
\cite{Duvaut-Lions:1972,Sermange-Temam:1983,Miao-Yuan-Zhang:2007,Renardy:2011,He-Huang-Wang:2014,Fukumoto-Zhao:2019}
and the citations therein.
Some of these results concerning existence, uniqueness and stability
of the solution fields $(\uv,\ps,\Ev,\Bv)$ have been derived under
specific assumptions that may reduce the generality of the model.
To the best of our knowledge, the full mathematical understanding of
the MHD model is still an open issue and an active research area.
This topic is beyond the scope of our work; thus, we shall simply
assume that the MHD model is well-posed, at least in the setting that
we are using in this paper.


\section{Virtual element spaces for the incompressible flow equations}
\label{section:VEM:Navier-Stokes}

\subsection{The discrete pressure space}
\label{subsection:pressure-space}
We approximate the pressure unknown in the space of piecewise
discontinuous, constant functions with zero average on $\Omega$:
\begin{align*}
  \Qsh
  := \Big\{ \qsh\in\LSzr{2}(\Omega)\,\mid\,\restrict{\qsh}{\P}\in\PS{0}(\P) \; \forall\P\in\Th   \Big\} 
  =  \PS{0}(\Th)\cap\LSzr{2}(\Omega).
\end{align*}
Every $\qsh\in\Qsh$ is uniquely determined by the set of constant
values $\big(\restrict{\qsh}{\P}\big)$.
Accordingly, we can approximate any scalar function
$\qs\in\LSzr{2}(\Omega)$ by its averages over the mesh elements.

\subsection{The discrete velocity space}
\label{subsection:velocity-space}

Here, we consider the ``lowest order'' version of the spaces
introduced
in~\cite{BeiraodaVeiga-Lovadina-Vacca:2018,BeiraodaVeiga-Dassi-Vacca:2020}.
This includes an ``enhancement'' procedure that allows for the
computation of an $L^2$ linear polynomial projector.
We refer to such papers for a better understanding of the motivations
behind this construction.

We define the nodal velocity space on a face~$\F\in\setFace$ as
\begin{align*}
  \Wvh(\F) := \Big\{
  \wvh\in\big[\HS{1}(\F)\big]^3
  \,\mid\,
  &
  \restrict{\wvh}{\partial\F}\in\big[\CS{0}(\partial\F)\big]^3,\;
  \restrict{\wvh}{\e}\in\big[\PS{1}(\e)\big]^3\,\forall\e\in\setEdgeF,\quad
  \nonumber\\
  &
  \DeltaF\wvh\in\big[\PS{2}(\F)\big]^3,
  \nonumber\\
  &
  \big( \wv_{\hh,\taubold} - \PiboldnablaF_1 \wv_{\hh,\taubold}, \qv \big)_{0,\F} = 0 \quad\forall\qv\in\big[\PS{2}(\F)\big]^2,\nonumber\\
  &
  \big( \ws_{\hh,\nv} - \PinablaF_1 \ws_{\hh,\nv}, \qs \big)_{0,\F} = 0 \quad\forall\qs\in\PS{2}(\F)\setminus\Rbb
  \Big\}.
\end{align*}
Above, we denoted the tangential and normal components of a given
field~$\wv$ by~$\wv_{\hh,\taubold}=(\nvF\times\wv)\times\nvF$
and~$\ws_{\hh,\nv}=\nvF\cdot\wv$, respectively.
We define the nodal velocity space on the whole boundary~$\partial\P$
of a mesh element~$\P$ as
\[
  \Wvh(\partial\P)
  := \Big\{ \wvh\in\big[\CS{0}(\partial\P)\big]^3\,\mid\,\restrict{\wvh}{\F}\in\Wvh(\F)\quad\forall\F\in\setFaceP \Big\}.
\]
Finally, we introduce the nodal velocity space on the
element~$\P\in\Th$ as follows:
\begin{align*}
  \Wvh(\P) :=
  \Big\{\,
  \wvh\in\big[\HS{1}(\P)\big]^3
  \,\mid \,
  &
  \restrict{\wvh}{\partial\P}\in\Wvh(\partial\P),
  \quad 
  \nonumber\\
  &
  \begin{cases}
  \Delta\wvh + \nabla\ss\in\xv\times\big[\PS{0}(\P)\big]^3\text{~for~some~}\ss\in\LSzr{2}(\P),\\
  \div\wvh\in\PS{0}(\P),
  \end{cases}
  \nonumber\\
  &
  \big(\wvh - \PiboldnablaE\wvh, \xv\times\qv \big)_{0,\E} = 0 \quad\forall\qv\in\big[\PS{0}(\P)\big]^3
  \,\Big\}.
\end{align*}
Every virtual element vector field $\wvh\in\Wvh(\P)$ is uniquely
determined by the set of values
$\big(\,(\wvV)_{\V\in\setNodeP},(\wsF)_{\F\in\setFaceP}\,\big)$, where
\begin{itemize}
\item $\wvV=\wvh(\xvV)$ is the value of~$\wvh$ at the vertex
  $\V\in\setNodeP$ (with position vector $\xvV$);
\item $\wsF$ is the scaled zero-th order moment of the normal component of
  $\wvh$ on the face~$\F\in\setFaceP$, given by
  \begin{align*}
    \wsF = \frac{1}{\mF}\int_\F\ws_{\hh,\nv}.
  \end{align*}
\end{itemize}
The unisolvence of these degrees of freedom can be proved as in~\cite{Antonietti-BeiraodaVeiga-Mora-Verani:2014, BeiraodaVeiga-Dassi-Vacca:2020}.
For any~$\wvh\in\Wvh(\P)$, the elliptic projection
$\PinvP{1}\wvh\in\big[\PS{1}(\P)\big]^{3}$, the $\LS{2}$-orthogonal
projection $\PizvP{1}\wvh\in\big[\PS{1}(\P)\big]^3$, and the
divergence of $\wvh$ are directly computable from such degrees of
freedom.
Finally, we define the global nodal velocity space by an
$\HS{1}$-conforming coupling of the local degrees of freedom:
\begin{align*}
  \Wvh:=\Big\{\,\wvh\in\big[\HSzr{1}(\Omega)\big]^3\,\mid\,\restrict{\wvh}{\P}\in\Wvh(\P)\quad\forall\P\in\Th\,\Big\}.
\end{align*}

\begin{rem} \label{remark:inf-sup}
As in~\cite{BeiraodaVeiga-Lovadina-Vacca:2018, BeiraodaVeiga-Dassi-Vacca:2020}, the couple of spaces~$\Wvh \times \Qsh$ satisfies a discrete inf-sup condition, i.e., there exists~$\beta>0$ independent of~$\hh$ such that
\[
\beta \le
\inf_{\qsh \in \Qsh}\sup_{\vvh \in \Wvh} \frac{b(\vvh,\qsh)}{\norm{\vvh}{1} \norm{\qsh}{}}.
\]
Furthermore, we also have the important property
\[
\div \Wvh \subseteq \Qsh.
\]
\end{rem}

\subsection{The virtual element bilinear forms $\ash(\cdot,\cdot)$ and $\msh(\cdot,\cdot)$}
\medskip
We use the elliptic projection operator $\PinvP{1}$ to define the
elemental bilinear form $\ashP:\Wvh(\P)\times\Wvh(\P)\to\Rbb$, which
mimics the $\HS{1}$-inner product on the element $\P$:
\begin{multline}   
  \ashP(\uvh,\vvh) :=
  \big(\nabla\PinvP{1}\uvh,\nabla\PinvP{1}\vvh\big)_{\P} +
  \SPa{}\big( (\Is-\PinvP{1})\uvh, (\Is-\PinvP{1})\vvh \big)\\
  \quad\forall\uvh,\,\vvh\in\Wvh(\P).
  \label{eq:velocity:H1:local:bilform}
\end{multline}
Here, $\SPa{}:\Wvh(\P)\times\Wvh(\P)\to\Rbb$ can be any computable symmetric bilinear form such that there exist two positive constants~$\sigma_*$ and~$\sigma^*$ independent of~$\hE$ satisfying
\begin{align*}
  \sigma_*\vert\vvh\vert^2_{1,\P}
  \leq\SPa{}(\vvh,\vvh)\leq
  \sigma^*\vert\vvh\vert^2_{1,\P}
  \quad\quad\forall\vvh\in\ker\big(\PinvP{1}\big)\cap\Wvh(\P).
\end{align*}
A ``standard'' choice for the stabilization is given by~\cite{BeiraodaVeiga-Dassi-Vacca:2020}
\begin{align*}
  \SPa{}(\uvh,\vvh) := \hP\sum_i \dofiP(\uvh)\,\dofiP(\vvh),
  \qquad\forall\uvh,\,\vvh \in \Wvh(\P),
\end{align*}
where the summation on the index $i$ is carried over all the elemental
degrees of freedom and $\dofiP(\wvh)$ is the bounded, linear
functional defined on $\Wvh(\P)$ providing the $i$-th degree of
freedom of $\wvh\in\Wvh(\P)$.
The summation term in $\SPa{}(\cdot,\cdot)$ is multiplied by $\hP$ to
have a consistent scaling for both terms of $\ashP(\cdot,\cdot)$ on
the right-hand side of~\eqref{eq:velocity:H1:local:bilform} with respect to the
element size.

The local bilinear form $\ashP(\cdot,\cdot)$ satisfies the two
fundamental properties of consistency and stability:
\begin{itemize}
\item \textbf{consistency}: for all $\qv\in\big[\PS{1}(\P)\big]^3$ and $\wvh\in\Wvh(\P)$,
\begin{align} \label{consistency:velocity:H1}
\ashP(\qv,\wvh) = \asP(\qv,\wvh);
\end{align}
\item \textbf{stability}: for all $\wvh\in\Wvh(\P)$,
  \begin{align}
    \alpha_*\norm{\wvh}{1,\P}^2
    \leq\ashP(\wvh,\wvh)\leq
    \alpha^*\norm{\wvh}{1,\P}^2,
    \label{eq:velocity:H1:stability}
  \end{align}
  with $\alpha_*=\min(1,\sigma_*)$ and $\alpha^*=\max(1,\sigma^*)$.
\end{itemize}
Finally, we define the global bilinear form
$\ash:\Wvh\times\Wvh\to\Rbb$ that we use in the virtual element
formulation of the MHD model:
\begin{align*}
  \ash(\uvh,\vvh) := \sum_{\P\in\Th} \ashP(\uvh,\vvh)
  \quad\quad\forall\uvh,\vvh\in\Wvh.
\end{align*}
In what follows, we use the discrete norm
\[
\norm{\vvh}{\Wvh}^2 := \ash(\vvh,\vvh) \qquad \forall \vvh \in \Wvh.
\]

\medskip\noindent
We use the $\LS{2}$-projection operator $\PizvP{1}$ to define the
bilinear form $\mshP:\Wvh(\P)\times\Wvh(\P)\to\Rbb$, which mimics the
local $\LS{2}$-inner product on the element $\P$:
\begin{align}
  \mshP(\uvh,\vvh) :=
  \big(\PizvP{1}\uvh,\PizvP{1}\vvh\big)_{\P} +
  \SPm{}\Big( (\Is - \PizvP{1})\uvh, (\Is - \PizvP{1})\vvh \Big)
  \quad\forall\uvh,\,\vvh\in\Wvh(\P).
  \label{eq:velocity:L2:local:bilform}
\end{align}
Here, $\SPm{}:\Wvh(\P)\times\Wvh(\P)\to\Rbb$ can be any computable bilinear form such that there exist two positive
constants~$\widetilde{\sigma}_*$ and~$\widetilde{\sigma}^*$ independent of~$\hE$ satisfying
\[
  \widetilde{\sigma}_*\Vert\vvh\Vert^2_{0,\P}
  \leq\SPm{}(\vvh,\vvh)\leq
  \widetilde{\sigma}^*\Vert\vvh\Vert^2_{0,\P}
  \quad\quad\forall\vvh\in\ker\big(\PizvP{1}\big)\cap\Wvh(\P).
\]
A ``standard'' choice for the stabilization term~$\SPm{}(\cdot,\cdot)$
is given by:
\begin{align*}
  \SPm{}(\uvh,\vvh) := \hP^3\sum_i \dofiP(\uvh)\,\dofiP(\vvh),
  \qquad \forall \uvh, \,\vvh\in \Wvh(\P),
\end{align*}
where the summation on the index $i$ is again on all the elemental
degrees of freedom and $\dofiP(\cdot)$ is the same functional used in
the definition of the stabilization term $\SPa{}(\cdot,\cdot)$.
The summation term in $\SPm{}(\cdot,\cdot)$ is multiplied by $\hP^3$
to have a consistent scaling for both terms of $\mshP(\cdot,\cdot)$ on the right-hand side of~\eqref{eq:velocity:L2:local:bilform}
with respect to the element size.
The stability bounds above are shown in~\cite{Beirao-Mascotto-Meng:2022}.

The local bilinear form $\mshP(\cdot,\cdot)$ satisfies the two
fundamental properties of consistency and stability:
\begin{itemize}
\item \textbf{consistency}: for all $\qv\in\big[\PS{1}(\P)\big]^3$ and
  $\wvh\in\Wvh(\P)$,
  \begin{align}
    \mshP(\qv,\wvh) = (\qv,\wvh)_{0,\P};
    \label{eq:velocity:L2:consistency}
  \end{align}
\item \textbf{stability}: for all $\wvh\in\Wvh(\P)$,
  \begin{align}
    \mu_*\Vert\wvh\Vert^2_{0,\P}
    \leq\mshP(\wvh,\wvh)\leq
    \mu^*\Vert\wvh\Vert^2_{0,\P},
    \label{eq:velocity:L2:stability}
  \end{align}
  with $\mu_*=\min(1,\widetilde{\sigma}_*)$ and
  $\mu^*=\max(1,\widetilde{\sigma}^*)$.
\end{itemize}
Finally, we define the global inner product
$\msh:\Wvh\times\Wvh\to\Rbb$ that we use in the virtual element
formulation of the MHD model:
\begin{align*}
  \msh(\uvh,\vvh) := \sum_{\P\in\Th} \mshP(\uvh,\vvh).
\end{align*}
Associated with~$\msh(\cdot,\cdot)$, we define the corresponding discrete norm
\[
\norm{\vvh}{\msh}^2:= \msh(\vvh,\vvh).
\]

\subsubsection{Interpolation in~$\Wvh$ and a stability result} \label{subsubsection:interpolation-velocity}
We define the (energy) interpolant $\uvI\in\Wvh$ of a sufficiently smooth vector-valued field $\uv$ as the unique function in~$\Wvh$ sharing the same degrees of freedom of~$\uv$.
We have the following local and global interpolation property;
see~\cite{Beirao-Mascotto-Meng:2022}.

\begin{lem} \label{lemma:uvI:intp}
Let~$\uv\in\big[\HS{2}(\Omega)\big]^3$ and $\uvI\in\Wvh$ be its degrees of freedom interpolant.
Then, a real, positive constant $\Cintp$ independent of $\hh$ exists such that
\[
\norm{\uv-\uvI}{0,\E} + \hE\snorm{\uv-\uvI}{1,\E} \leq \Cintp \hE^2 \snorm{\uv}{2,\E} \quad \forall \E \in \Th;
\quad
\norm{\uv-\uvI}{} + \hh \norm{\uv-\uvI}{1} \leq \Cintp \hh^2 \snorm{\uv}{2}.
\]
The constant $\Cintp$ depends on the mesh parameter~$\gamma$.
\end{lem}
An important consequence of the definition of the interpolant~$\uvI$
is that, if~$\uv$ is divergence free, then (see also
Remark~\ref{remark:inf-sup})
\begin{equation} \label{uI-div}
    \bs(\uvI,\qsh) = \bs(\uv,\qsh) = 0 \qquad \forall \qsh \in \Qsh \quad \Longrightarrow \quad \div \uvI =0.
\end{equation}
We conclude this section with a technical lemma that will be useful in Section~\ref{section:VEM:analysis}.
\begin{lem} \label{lemma:chi:error}
We have a stability bound for the $\LS{\infty}$-norm of the elemental $L^2$ polynomial projection of~$\uvI$: for all~$\ell \in \Nbb$,
\[
\quad\norm{\Pizv{\ell} \uvI}{\LS{\infty}}
\leq \Cinv \CD \big(\CI \snorm{\uv}{W^{1,3}} + \norm{\uv}{L^{\infty}}).
\]
\end{lem}
\begin{proof}
On each element~$\E \in \Th$, standard manipulations imply
\begin{align*}
\norm{\Pizv{\ell} \uvI}{\LS{\infty}(\P)}
& \leq \norm{\Pizv{\ell}(\uv-\uvI)}{\LS{\infty}(\P)} + \norm{\Pizv{\ell}\uv}{\LS{\infty}(\P)}
\leq \Cinv \hP^{-\frac32}\big(\norm{\uv-\uvI}{0,\P} + \norm{\uv}{0,\P}\big)\nonumber\\[0.5em]
& \leq \Cinv \big(\CI\hP^{-\frac12} \snorm{\uv}{1,\P} + \hP^{-\frac32} \norm{\uv}{0,\P} \Big)
  \le  \Cinv \big(\CI \CD \snorm{\uv}{\WS{1,3}(\P)} + \CD \norm{\uv}{\LS{\infty}(\P)} \big).
\end{align*}
Taking the maximum over all elements gives the assertion.
\end{proof}

\subsection{The discrete trilinear forms}
We introduce the local continuous and discrete trilinear forms on the element~$\E$:
\begin{align*}
  \cE(\vv;\uv,\wv)     &:= \int_{\P} (\nabla \uv) \vv\cdot\wv \qquad \forall\uv,\vv,\wv\in\big[\HS{1}(\P)\big]^3, \\
  \chE(\vvh;\uvh,\wvh) &:= \int_{\P} (\PizvP{0}\nabla\uvh)\PizvP{1}\vvh\cdot\PizvP{1}\wvh \qquad \forall\uvh,\vvh,\wvh\in\Wvh(\P).
\end{align*}
We also define their skew-symmetric counterparts
\begin{align}
  \ctildeE(\vv;\uv,\wv)     &:= \frac12 \left( \cE(\vv;\uv,\wv)     - \cE(\vv;\wv,\uv)     \right) \qquad \forall \uv,\vv,\wv\in\big[\HS{1}(\P)\big]^3,\label{eq:ctildeE:def}\\
  \ctildenE(\vvh;\uvh,\wvh) &:= \frac12 \left( \cnE(\vvh;\uvh,\wvh) - \cnE(\vvh;\wvh,\uvh) \right) \qquad \forall \uvh,\vvh,\wvh\in\Wvh(\P),\label{eq:ctildeEh:def}
\end{align}
and the associated global skew-symmetric trilinear forms
\begin{align*}
  \ctilde(\vv;\uv,\wv)     &= \sum_{\P \in \Th} \ctildeE(\vv;\uv,\wv)     \qquad\forall \uv,\vv,\wv\in\big[\HS{1}(\Omega)\big]^3,\\
  \ctilden(\vvh;\uvh,\wvh) &= \sum_{\P \in \Th} \ctildenE(\vvh;\uvh,\wvh) \qquad\forall \uvh,\vvh,\wvh\in\Wvh.
\end{align*}
If~$\vv \in [H^1_0(\Omega)]^3$ is divergence free, then
\[
\ctilde(\vv; \uv,\wv) = c(\vv; \uv,\wv).
\]
In the remainder of the section, we discuss three properties of these trilinear forms.
For the sake of exposition, we consider the extension of the virtual
element bilinear and trilinear forms to $\big[\HS{1}(\P)\big]^3$ and
$\big[\HS{1}(\Omega)\big]^3$.
\medskip

The first property is the continuity of the trilinear form, which can be proved as in~\cite[Proposition~$3.3$]{BeiraodaVeiga-Lovadina-Vacca:2018}.
\begin{lem} \label{lemma:continuity-trilinear}
The following continuity bound is valid:
\begin{align*}
    \ABS{ \ctilden(\vvh; \uvh, \wvh) }
    \lesssim 
    \snorm{\uvh}{1}\,\snorm{\vvh}{1}\,\snorm{\wvh}{1},
\end{align*}
where the hidden constant is independent of~$\hh$.
\end{lem}

\medskip
The second property is the measure of the variational crime
perpetrated in the discretization of the trilinear form.
The proof is a modification of the proof
of~\cite[Lemma~$4.3$]{BeiraodaVeiga-Lovadina-Vacca:2018}.

\begin{lem} \label{lemma:property-trilinear2}
Given~$\wv\in\big[\HSzr{1}(\Omega)\big]^3$ and $\vv\in\big[\HSzr{1}(\Omega)\big]^3\cap\big[\HS{2}(\Omega)\big]^3$,
we have the following bound
\begin{align*}
\ABS{ \cst(\vv;\vv,\wv) - \cht(\vv;\vv,\wv) }  \le \CC\hh \norm{\vv}{2}^2 \, \snorm{\wv}{1},
\end{align*}
where the constant~$\CC$ depends only on the constant of inverse estimates, approximation bounds, and Sobolev embedding inequalities, but is independent of~$\hE$.
\end{lem}
\begin{proof}
  A straightforward manipulation of
  relations~\eqref{eq:ctildeE:def}-\eqref{eq:ctildeEh:def} defining
  the skew-symmetric trilinear forms yields
\begin{equation} \label{mu1mu2}
\begin{split}
\ABS{ \ctilde(\vv; \vv, \wv) - \ctilden(\vv; \vv,\wv) }
& \le \frac12\left( \ABS{ c(\vv; \vv, \wv) - \cn(\vv;\vv,\wv) }
                    + \ABS{ c(\vv; \wv, \vv) - \cn(\vv;\wv,\vv) }\right) \\
& = \frac12 \left( \TERM{T}{1} + \TERM{T}{2}  \right).
\end{split}
\end{equation}
We estimate the two terms on the right-hand side separately. We begin with the first one.
Using the definition of the orthogonal projection $\Pizv{1}\vv$, and adding and subtracting $(\nabla\vv)\vv\cdot\Pizv{1}\wv$ to the integral argument, we write
\begin{align*}
\begin{split}
\TERM{T}{1}
& =   \ABS{ \int_{\Omega} (\nabla\vv)\vv \cdot \wv - \int_{\Omega} (\Pizv{0}\nabla\vv)\Pizv{1}\vv \cdot \Pizv{1}\wv }\\[0.5em]
& =   \ABS{ \int_{\Omega} (\nabla\vv)\vv \cdot \wv - \int_{\Omega} (\Pizv{0}\nabla\vv)\vv \cdot \Pizv{1}\wv }\\[0.5em]
& \le \ABS{ \int_{\Omega} (\nabla\vv)\vv \cdot (\wv - \Pizv{1} \wv) } 
      + \ABS{ \int_{\Omega} (\nabla\vv - \Pizv{0}\nabla\vv)\vv\cdot\Pizv{1}\wv}
   =: \TERM{T}{1,1} + \TERM{T}{1,2}.
\end{split}
\end{align*}
Again, we focus on the two terms on the right-hand side separately.
Using the H{\"o}lder inequality, polynomial approximation properties, and the Sobolev embedding theorem gives
\begin{align*}
\begin{split}
\TERM{T}{1,1}
& \le \norm{\nabla\vv}{\LS{4}}\,\norm{\vv}{\LS{4}}\,\norm{\wv - \Pizv{1} \wv}{}
\le \Cappr \hh \norm{\nabla \vv}{\LS{4}}\,\norm{\vv}{\LS{4}} \norm{\wv}{1}
\le \CSob^2 \Cappr \hh \norm{\vv}{1}\,\norm{\vv}{2} \norm{\wv}{1}.
\end{split}
\end{align*}
Next, for each element~$\E \in \Th$,
we apply a polynomial inverse inequality, the H{\"o}lder inequality, and standard manipulations:
\[
\norm{\Pizv{1} \wv}{\LS{4}(\P)}
\le \Cinv \hP^{-\frac34} \norm{\wv}{0,\P} 
\le \Cinv \hP^{-\frac34} \norm{1}{\LS{4}(\P)}\,\norm{\wv}{\LS{4}(\P)}
\le \Cinv \CD \norm{\wv}{\LS{4}(\P)}.
\]
This entails
\begin{align*}
\TERM{T}{1,2}
& \le \sum_{\E\in\Th} \norm{\nabla\vv -\Pizv{0}\nabla\vv}{0,\P}\,\norm{\vv}{\LS{4}(\P)}\,\norm{\Pizv{1}\wv}{\LS{4}(\P)} \nonumber\\[0.5em]
&\le \Cinv \CD \Cappr\hh \sum_{\P\in\Th} \snorm{\vv}{2,\P}\,\norm{\vv}{\LS{4}(\P)}\, \norm{\wv}{\LS{4}(\P)}.
\end{align*}
We apply a sequential $\ell^2-\ell^4-\ell^4$ H{\"o}lder inequality and the Sobolev embedding theorem twice, and deduce
\[
\TERM{T}{1,2}
\le  \Cinv\CD\Cappr \hh\snorm{\vv}{2}\,\norm{\vv}{\LS{4}}\, \norm{\wv}{\LS{4}}
\le  \Cappr\Cinv\CD\CSob^2 \hh \norm{\vv}{2}\, \norm{\vv}{1}\, \norm{\wv}{1}.
\]
This step concludes the derivation of a bound on the term~$\TERM{T}{1}$ in~\eqref{mu1mu2}.

Next, we focus on the term~$\TERM{T}{2}$.
Adding and subtracting $\big(\Pizv{0}\nabla\wv\big)\vv\cdot\vv$ yields
\begin{align*}
\TERM{T}{2}
& =   \ABS{\int_{\Omega} (\nabla\wv)\vv\cdot\vv - \int_{\Omega}(\Pizv{0}\nabla\wv)\vv\cdot\Pizv{1}\vv }\\
& \le \ABS{ \int_{\Omega} (\nabla\wv - \Pizv{0}\nabla\wv)\vv\cdot\vv }
      + \ABS{ \int_{\Omega} (\Pizv{0}\nabla\wv)\vv\cdot(\vv - \Pizv{1}\vv) }
  =: \TERM{T}{2,1} + \TERM{T}{2,2}.
\end{align*}
We derive the bounds on the terms~$\TERM{T}{2,1}$ and~$\TERM{T}{2,2}$ by using the approximation properties of the orthogonal projections, and theoretical tools similar to those used for the bound on~$\TERM{T}{1}$:
\begin{align*}
\TERM{T}{2,1}
& = \ABS{\int_{\Omega} (\nabla\wv) (\vv\cdot\vv - \Pizv{0}\vv\cdot\vv)}
  \le \snorm{\wv}{1}\ \norm{\vv\cdot\vv - \Pizv{0}(\vv\cdot\vv)}{}
  \le \Cappr\hh \snorm{\wv}{1}\ \snorm{\vv\cdot\vv}{1}\\
& \le \Cappr\hh \snorm{\wv}{1}\ \norm{\vv}{\WS{1,4}}^2 
  \le \Cappr\CSob^2\hh  \norm{\vv}{2}^2 \snorm{\wv}{1}.
\end{align*}
On the other hand, given~$\vv_\pi$ the best piecewise linear approximant in~$L^4$ of~$\vv$ over~$\Th$, we also have
\begin{align*}
\TERM{T}{2,2}
& \le \sum_{\E \in \Th} \snorm{\wv}{1,\P}\,\norm{\vv}{\LS{4}(\P)}\,\norm{\vv - \Pizv{1}\vv}{\LS{4}(\P)}\\
& \le \sum_{\E \in \Th} \snorm{\wv}{1,\P}\,\norm{\vv}{\LS{4}(\P)} \left( \norm{\vv-\vv_\pi}{\LS{4}(\P)} + \norm{\vv_\pi - \Pizv{1}\vv}{\LS{4}(\P)} \right).
\end{align*}
Observe that
\begin{align*}
\norm{\vv_{\pi}-\PizvP{1}\vv}{\LS{4}(\P)}
& \le \Cinv \hP^{-\frac34} (\norm{\vv-\vv_\pi}{0,\P} + \norm{\vv-\PizvP{1}\vv}{0,\P})\\
& \le 2 \Cinv \hP^{-\frac34} \norm{\vv-\vv_\pi}{0,\P}
  \le 2 \Cinv \CD \norm{\vv-\vv_\pi}{\LS{4}(\P)} \\
& \le 2 \Cinv \CD \Cappr \hP \snorm{\vv}{\WS{1,4}(\P)} .
\end{align*}
Inserting this bound above and using polynomial approximation estimates yield
\[
\TERM{T}{2,2} \le (1 + 2 \Cinv \CD) \Cappr \hh 
\sum_{\E \in \Th} \snorm{\wv}{1,\P} \norm{\vv}{\LS{4}(\P)} \norm{\vv}{\WS{1,4}(\P)}.
\]
We apply a sequential $\ell^2-\ell^4-\ell^4$ H{\"o}lder inequality and the Sobolev embedding theorem, and deduce
\[
\TERM{T}{2,2}
\le (1 + 2 \Cinv \CD) \Cappr \hh \snorm{\wv}{1} \norm{\vv}{\LS{4}} \norm{\vv}{\WS{1,4}}
\le (1 + 2 \Cinv \CD) \Cappr \CSob^2 \hh \norm{\vv}{1}  \norm{\vv}{2} \norm{\wv}{1}.
\]
Collecting all the bounds together yields the assertion.
\end{proof}

\medskip
The last property measures the distance between the continuous and the
discrete solutions through the discrete trilinear form.
The proof differs from that of~\cite[Lemma~$4.4$]{BeiraodaVeiga-Lovadina-Vacca:2018},
since here we are interested in the analysis of the time-dependent case.
For this reason, we also use some techniques from~\cite[Lemma~$4.1$]{Schroeder-Lube:2017}.
Further, we require some extra regularity on the exact velocity solution~$\uv$ to~\eqref{eq:continuous:problem-strong}.

\begin{lem} \label{lemma:property-trilinear3}
Let~$\uv \in W^{2,3}(\Omega)$ and~$\evhu := \uvh- \uvI$, where~$\uvI$ is the degrees of freedom interpolant of~$\uv$; see Lemma~\ref{lemma:uvI:intp}.
Then, for every positive~$\varepsilon$,
there exist two positive constants~$C_1$ and~$C_2$ independent of~$\hh$ such that
\begin{align*}
& \ABS{ \cht(\uv;\uv,\evhu) - \cht(\uvh;\uvh,\evhu) } 
  \leq  \varepsilon \snorm{\evhu}{1}^2
  + \left(1+\frac{1}{\varepsilon}\right)\Cs_1 R_1(\uv)  \norm{\evhu}{}^2
  + \left(1+\frac{1}{\varepsilon}\right)\Cs_2 R_2(\uv)  \hh^2  ,
\end{align*}
where
\[
R_1(\uv):= \norm{\uv}{\WS{2,3}}^2 + \norm{\uv}{\WS{2,3}} + \norm{\uv}{\WS{1,\infty}} + 1,
\qquad
R_2(\uv):= \norm{\uv}{\WS{2,3}}^4 + \norm{\uv}{\WS{2,3}}^2 \snorm{\uv}{\WS{1,\infty}}^2.
\]
\end{lem}
\begin{proof}
The triangle inequality implies
\begin{align*}
& \ABS{ \ctilden(\uv;\uv,\evhu) - \ctilden(\uvh;\uvh,\evhu)} \nonumber\\
& \le \ABS{ \ctilden(\uv;\uv,\evhu)  - \ctilden(\uvh;\uv,\evhu)  } +   \ABS{ \ctilden(\uvh;\uv,\evhu) - \ctilden(\uvh;\uvh,\evhu) }
  =: \TERM{T}{1} + \TERM{T}{2}.
\end{align*}
We estimate the two terms on the right-hand side separately and begin by splitting~$\TERM{T}1$ into the sum of two further contributions, due to the definition of the skew symmetric form~$\ctilden(\cdot;\cdot,\cdot)$:
\[
\TERM{T}{1}
\le \frac12 (\ABS{\cn(\uv;\uv,\evhu) - \cn(\uvh;\uv,\evhu)} + \ABS{\cn(\uv;\evhu,\uv) - \cn(\uvh;\evhu,\uv)})
=: \frac12 (\TERM{T}{1,1} + \TERM{T}{1,2}).
\]
Using standard manipulations, we get
\begin{align*}
\TERM{T}{1,1}
& = \ABS{ \int_\Omega(\Pizv{0}\nabla\uv)(\uv-\uvh)\cdot\Pizv{1}\evhu}
    \le \norm{\Pizv{0}\nabla\uv}{\LS{\infty}}\,\norm{\uv - \uvh}{}\,\norm{\Pizv{1}\evhu}{}\\
& \le \snorm{\uv}{\WS{1,\infty}} \left( \norm{\uv - \uvI}{} + \norm{\evhu}{} \right) \norm{\evhu}{}
      \le \CI \hh \snorm{\uv}{\WS{1,\infty}}\,\snorm{\uv}{1}\,\norm{\evhu}{} + \snorm{\uv}{\WS{1,\infty}}\,\norm{\evhu}{}^2\\
& \le \CI \hh^2 \snorm{\uv}{\WS{1,\infty}}^2 \,\norm{\uv}{1}^2 
  + (\snorm{\uv}{\WS{1,\infty}}  +\CI) \norm{\evhu}{}^2\\
& \le \CI \CSob^2 \hh^2 \snorm{\uv}{\WS{1,\infty}}^2\norm{\uv}{\WS{2,3}}^2
  + (\snorm{\uv}{\WS{1,\infty}}  +\CI) \norm{\evhu}{}^2.
\end{align*}
As for the term~$\TERM{T}{1,2}$, we write
\begin{align*}
\TERM{T}{1,2}
& = \ABS{ \int_{\Omega} (\Pizv{0}\nabla\evhu)(\Pizv{1}(\uv - \uvh))\cdot\uv }
    \le  \snorm{\evhu}{1}\,\norm{\uv -\uvh}{}\,\norm{\uv}{\LS{\infty}}\\
&\le \snorm{\evhu}{1}\,\left( \norm{\uv -\uvI}{} + \norm{\evhu}{} \right) \norm{\uv}{\LS{\infty}}
 \le \CI\hh\snorm{\evhu}{1}\,\snorm{\uv}{1}\,\norm{\uv}{\LS{\infty}}
    + \snorm{\evhu}{1}\,\norm{\evhu}{}\,\norm{\uv}{\LS{\infty}}\\
& \le \frac\varepsilon2 \snorm{\evhu}{1}^2 + \frac1\varepsilon \CI^2\hh^2\snorm{\uv}{1}^2\,\norm{\uv}{\LS{\infty}}^2 + \frac1\varepsilon \norm{\uv}{\LS{\infty}}^2\,\norm{\evhu}{}^2\\
& \le \frac\varepsilon2 \snorm{\evhu}{1}^2 + \frac1\varepsilon \CI^2 \CSob^2 \hh^2 \norm{\uv}{\WS{2,3}}^4 + \frac1\varepsilon \norm{\uv}{\LS{\infty}}^2\,\norm{\evhu}{}^2.
\end{align*}
Next, we focus on the term~$\TERM{T}2$. Note that
\begin{equation} \label{note-that}
    \ctildenE(\uvh;\uvh,\evhu) = \ctildenE(\uvh;\uvI,\evhu),
\end{equation}
which is a consequence of the skew symmetry of~$\ctildenE(\cdot;\cdot,\cdot)$ and~\eqref{uI-div}.

We split the term~$\TERM{T}2$ into two contributions, due to the definition of the skew symmetric form~$\ctilden(\cdot;\cdot,\cdot)$. Recalling~\eqref{note-that}, we write
\begin{equation} \label{T21T22}
\TERM{T}2
\le \frac12 \left(\ABS{\cn(\uvh;\uv,\evhu) - \cn(\uvh;\uvI,\evhu) } + \ABS{\cn(\uvh;\evhu,\uv) - \cn(\uvh;\evhu,\uvI)} \right)
=: \frac12( \TERM{T}{2,1} + \TERM{T}{2,2}).
\end{equation}
Using standard manipulations as above, we get
\begin{align*}
\TERM{T}{2,1}
\le \ABS{ \cn(\evhu;\uv-\uvI,\evhu) } + \ABS{\cn(\uvI;\uv-\uvI,\evhu)}
=: \TERM{T}{2,1,1} + \TERM{T}{2,1,2}. 
\end{align*}
Preliminary, for all~$\E \in \taun$, we observe that
\[
\begin{split}
\norm{\Pizv{0} \nabla (\uv-\uvI)}{\LS{\infty}(\P)}
& \le \CD \hP^{-\frac32} \snorm{\uv-\uvI}{1,\P}
\le \CD \CI \hP^{-\frac12} \snorm{\uv}{2,\P}
\le \CD^2 \CI \snorm{\uv}{W^{2,3}(\P)}.
\end{split}
\]
By taking the maximum over all the elements on both sides entails
\[
\norm{\Pizv{0} \nabla (\uv-\uvI)}{\LS{\infty}} \le \CD^2 \CI \snorm{\uv}{W^{2,3}}.
\]
Then, we have
\begin{align*}
\TERM{T}{2,1,1} 
&  \le \int_{\Omega} \ABS{ \Pizv{0}\nabla(\uv-\uvI) }\,\ABS{ \Pizv{1} \evhu }\,\ABS{ \Pizv{1}\evhu }
  \le \norm{ \Pizv{0}\nabla(\uv - \uvI) }{\LS{\infty}}\,\norm{\evhu}{}^2\\
& \le  \CD^2 \CI \snorm{\uv}{W^{2,3}} \,\norm{\evhu}{}^2.
\end{align*}
Next, we apply Lemma~\ref{lemma:chi:error} and write
\[
\norm{ \PizvP{1}\uvI }{\LS{\infty}(\P)}
\le \Cinv\CD \Big(\CI\snorm{\uv}{\WS{1,3}} + \norm{\uv}{\LS{\infty}}  \Big).
\]
By means of the above inequality, we deduce the bound
\begin{align*}
\TERM{T}{2,1,2}
& \le \int_{\Omega} \ABS{ \Pizv{0}\nabla(\uv-\uvI) }\,\ABS{ \Pizv{1}\uvI}\,\ABS{ \evhu }
\le \norm{\Pizv{0} \nabla( \uv- \uvI)}{} \, \norm{\Pizv{1}\uvI}{\LS{\infty}} \, \norm{\evhu}{}\\
& \le \CI \hh \snorm{\uv}{2} \Cinv\CD (\CI+1) (\norm{\uv}{\WS{1,3}}+\norm{\uv}{\LS{\infty}})  \norm{\evhu}{}\\
& \le \CI^2 \Cinv^2 \CD^2 (\CI+1)^2 \hh^2 \snorm{\uv}{2}^2 
(\norm{\uv}{\WS{1,3}}^2+\norm{\uv}{\LS{\infty}}^2) + \frac14\norm{\evhu}{}^2\\
& \le \CI^2 \Cinv^2 \CD^2 (\CI+1)^2 (\CSob^2 +1) \hh^2 
\norm{\uv}{\WS{2,3}}^4 + \frac14\norm{\evhu}{}^2.
\end{align*}
Next, we focus on the term~$\TERM{T}{2,2}$ appearing on the right-hand side of~\eqref{T21T22}:
\begin{align*}
\TERM{T}{2,2}
& \le \ABS{ \int_{\Omega} (\Pizv{0}\nabla\evhu)\uvh \cdot \Pizv{1}(\uv-\uvI) }\\
& \le \ABS{\int_{\Omega} (\Pizv{0}\nabla\evhu) \evhu \cdot\Pizv{1}(\uv -\uvI)}
 + \ABS{ \int_{\Omega} (\Pizv{0}\nabla\evhu)\uvI\cdot\Pizv{1}(\uv-\uvI)}
 =: \TERM{T}{2,2,1} + \TERM{T}{2,2,2}.
\end{align*}
We have the bounds
\begin{align*}
\TERM{T}{2,2,1} 
& \le \sum_{\E \in \Th} \norm{\PizvP{0} \nabla\evhu}{0,\P}\, \norm{\evhu}{0,\P}\,\norm{ \PizvP{1}(\uv - \uvI) }{\LS{\infty}(\P)}\\
& \le \sum_{\E \in \Th} \snorm{\evhu}{1,\P}\, \norm{\evhu}{0,\P}\, \Cinv\hP^{-\frac32}\, \norm{\uv-\uvI}{0,\P}\\
& \le \sum_{\E \in \Th} \snorm{\evhu}{1,\E} \norm{\evhu}{0,\E} \Cinv \CI \hP^{-\frac12} \snorm{\uv}{1,\E}
\le \Cinv \CI \CD \snorm{\uv}{\WS{1,3}} \sum_{\E\in\Th} \snorm{\evhu}{1,\P} \norm{\evhu}{0,\P}\\
& \le \Cinv \CI \CD \snorm{\uv}{\WS{1,3}}  \snorm{\evhu}{1} \norm{\evhu}{}
\le \frac\varepsilon4 \snorm{\evhu}{1}^2 + \frac{1}{\varepsilon} \Cinv^2 \CI^2 \CD^2 \,\snorm{\uv}{\WS{2,3}}^2\,\norm{\evhu}{}^2
\end{align*}
and, with similar computations,
\begin{align*}
\begin{split}
\TERM{T}{2,2,2}
& \le \sum_{\E \in \Th} \snorm{\evhu}{1,\P}\,\norm{\uvI}{0,\P}\,\norm{\PizvP{1}(\uv - \uvI)}{\LS{\infty}(\P)}\\
& \le \sum_{\E \in \Th} \snorm{\evhu}{1,\P}(1+\CI\hP)\norm{\uv}{1,\P}\, \Cinv\CI \hP^{\frac12} \snorm{\uv}{2,\P}\\ 
& \le \sum_{\E \in \Th} \snorm{\evhu}{1,\P} (1+\CI \hP) \norm{\uv}{1,\P}\,\Cinv\CI\CD \hP \snorm{\uv}{\WS{2,3}(\E)} \\
& \le \frac\varepsilon4 \snorm{\evhu}{1}^2 + \frac1\varepsilon (1+\CI\hP)^2\Cinv^2\CI^2 \CD^2\hh^2 \norm{\uv}{1}^2 \, \snorm{\uv}{\WS{2,3}}^2\\
& \le \frac\varepsilon4 \snorm{\evhu}{1}^2 + \frac1\varepsilon (1+\CI\hP)^2\Cinv^2\CI^2 \CD^2\hh^2 \norm{\uv}{\WS{2,3}}^4.
\end{split}
\end{align*}
We collect the bounds on all the terms above and the assertion follows.
\end{proof}


\section{Virtual element spaces for the electromagnetic equations}
\label{section:VEM:Maxwell}
  
\subsection{A pivot nodal space}
\label{subsection:nodal-space}

Here, we introduce a nodal virtual element space that will be
instrumental in defining polynomial projectors for edge spaces.
Consider the mesh face~$\F\in\calFh$ and the translated position vector $\xvF=\xv-\bvF$ for all $\xv\in\F$ (we recall that $\bvF$ is the center of~$\F$).
We define the nodal virtual element space on~$\F$ as
\begin{align*}
  \VshNode(\F) :=
  \bigg\{
  \vsh\in\CS{0}(\overline{\F})\,\mid\,\Delta\vsh\in\PS{0}(\F),\;
  \restrict{\vsh}{\e}\in\PS{1}(\e)\,\,\forall\e\in \setEdgeF,\;
  \int_\F\nabla\vsh\cdot\xvF = 0
  \bigg\}.
\end{align*}
Then, for all elements $\P\in\Th$, we define the local nodal space as:
\begin{align*}
  \VshNode(\P) :=
  \Big\{
  \vsh\in\CS{0}(\overline\P) \,\mid\,
  \Delta\vsh=0,\;
  \restrict{\vsh}{\F}\in\VshNode(\F)\,\,\forall\F\in \setFaceP
  \Big\}.
\end{align*}
The virtual element functions in the local space~$\VshNode(\P)$ for all $\P\in\Th$ are uniquely determined by the set of their vertex values over $\partial\P$;
see, e.g.,~\cite{BeiraodaVeiga-Brezzi-Dassi-Marini-Russo:2018} for a proof of their unisolvence.
We introduce the global nodal space
\begin{equation} \label{global-nodal-space}
\VshNode:=\{ \vsh \in \mathcal C^0(\overline \Omega)
             \mid \vsh{}_{|\P} \in \VshNode(\E) \; \forall \P \in \Th\},
\end{equation}
which will be used  in the definitions of the edge interpolant~$\EvII$ in the forthcoming Section~\ref{subsection:exact-sequences}.

\subsection{The edge space}
\label{subsection:edge-space}

We define the virtual element edge space on the mesh face
$\F\in\setFace$ as
\begin{align*}
  \VvhEdge(\F) :=
  \bigg\{
  \Evh\in\big[\LS{2}(\F)\big]^2
  \,\mid\,
  & \divF\Evh\in\PS{0}(\F),\;\rotF\Evh\in\PS{0}(\F),\nonumber\\
  & \hspace{1mm}\Evh\cdot\tv_{\e}\in\PS{0}(\e)\,\,\forall\e\in\setEdgeF,\;
  \int_\F\Evh\cdot\xvF = 0
  \bigg\},
\end{align*}
where~$\divF$ and $\rotF$ are defined in~\eqref{eq:2D:diff:ops}, and again $\xvF=\xv-\bv_{\F}$.

The local edge space on an element~$\P\in\Th$ is defined as
\begin{align*}
  \VvhEdge(\P) 
  := \bigg\{
  \Evh\in\big[\LS{2}(\E)\big]^3
  \,\mid\,\,
  &
  \Ev_{\hh,\tauvF}\in\VvhEdge(\F)\,\,\forall\F\in\setFaceP,\,\,
  \Evh\cdot\tv_{\e}\text{~continuous at each edge~$\e$,~}\nonumber\\
  &
  \begin{cases}
    \curl\curl\Evh\in\big[\PS{0}(\P)\big]^3\nonumber\\
    \div\Evh = 0
  \end{cases}
  \hspace{-4mm},
  \nonumber\\
  & \hspace{1mm}
  \int_{\P}\curl\Evh\cdot(\xvE\times\qv) = 0\,\,\forall\qv\in\big[\PS{0}(\P)\big]^3
  \bigg\},
\end{align*}
where $\xvE=\xv-\bvE$ for all~$\xv\in\E$, $\bvE$ is the center of~$\P$, and $\Ev_{\hh,\tauvF}=(\nvF\times\vvh)\times\nvF$.

\medskip
Each element $\Evh\in\VvhEdge(\P)$ is uniquely determined by the set
of constant values $(\Es_{\e})_{\e\in\setEdgeP}$, where
$\Es_{\e}=\Evh\cdot\tv_{\e}$ is the tangential component of $\Evh$
along the elemental edge $\e$.
We take these values as the degrees of freedom of $\Evh$ and refer the
reader to~\cite{BeiraodaVeiga-Brezzi-Dassi-Marini-Russo:2018} for the
proof of their unisolvence in $\VvhEdge(\P)$.
A major consequence of such a space definition is that $\PizvP{0}\Evh$
is computable from the degrees of freedom of $\Evh$ for all
$\Evh\in\VvhEdge(\P)$.

\medskip
We define the global edge space~$\VvhEdge$ by an
$\HS{}(\curl,\Omega)$-conforming coupling of the degrees of freedom,
so that
\begin{equation} \label{global-edge-space}
  \VvhEdge := \Big\{
  \Evh\in\HSzr{}(\curl,\Omega)\,\mid\,\restrict{\Evh}{\P}\in\VvhEdge(\P)\,\,\forall\P\in\Th  \Big\}.
\end{equation}
This definition includes the homogeneous tangential boundary
conditions on~$\partial\Omega$.

\subsubsection{Virtual element inner product}
According
to~\cite{BeiraodaVeiga-Brezzi-Dassi-Marini-Russo:2018,BeiraodaVeiga-Dassi-Manzini-Mascotto:2021},
we equip the virtual element edge space $\VvhEdge$ with an inner
product mimicking the $\LS{2}$ inner product.
Notably, we first introduce the local bilinear form
\begin{multline}
  \big[\Evh,\Fvh\big]_{\EDGE,\P} :=
  \Big(\PizvP{0}\Evh,\PizvP{0}\Fvh\Big)_{\P} +
  \SPEdge\Big( (\Is-\PizvP{0})\Evh, (\Is-\PizvP{0})\Fvh \Big)
  \\
  \forall\Evh,\,\Fvh\in\VvhEdge(\P),
  \label{discrete:bf:edge}
\end{multline}
where $\SPEdge(\cdot,\cdot)$ can be any computable, symmetric bilinear form such that there exist two positive constants~$\gamma_*$ and~$\gamma^*$ independent of~$\hE$ satisfying
\begin{align}
  \gamma_* \snorm{\Evh}{0,\P}^2
  \leq
  \SPEdge(\Evh,\Evh)  
  \le
  \gamma^* \snorm{\Evh}{0,\P}^2
  \quad\quad\forall\Evh\in\ker(\PizvP{0})\cap\VvhEdge(\P).
  \label{preliminary:stab}
\end{align}
An explicit stabilization satisfying~\eqref{preliminary:stab} was introduced in~\cite[formula~(4.8)]{BeiraodaVeiga-Brezzi-Dassi-Marini-Russo:2018}
and analyzed in~\cite[Proposition~5.5]{BeiraodaVeiga-Mascotto:2021}, and reads
\[
  \SPEdge(\Evh,\Fvh)  :=
  \hE^2\sum_{\F\in\setFaceP}\sum_{\e\in\setEdgeF}(\Evh\cdot\tv_{\e},\Fvh\cdot\tv_{\e})_{\e}
  \quad\quad \forall\Evh,\,\Fvh\in\VvhEdge(\P).
\]
The summation term in $\SPEdge(\cdot,\cdot)$ is multiplied by $\hP^2$
to have a consistent scaling for both terms of the edge inner
product~\eqref{discrete:bf:edge} with respect to the element size.
The bilinear form~\eqref{discrete:bf:edge} is computable on all
elements $\P\in\Th$ using the degrees of freedom of
$\Evh,\Fvh\in\VvhEdge(\P)$, and has the two crucial properties of
consistency and stability:
\begin{itemize}
\item \textbf{consistency}: for all $\qv\in\big[\PS{0}(\P)\big]^3$ and all $\Evh\in\VvhEdge(\P)$,
\begin{align} \label{consistency:edge}
\big[\qv,\Evh\big]_{\EDGE,\P} = (\qv,\Evh)_{\P};
\end{align}
  
\item \textbf{stability}: for all $\Evh\in\VvhEdge(\P)$,
  \begin{align}
    \eta_*\Vert\Evh\Vert^2_{0,\P}
    \leq
    \big[\Evh,\Evh\big]_{\EDGE,\P} 
    \leq
    \eta^*\Vert\Evh\Vert^2_{0,\P},
    \label{stability:edge}
  \end{align}
  with $\eta_*=\min(1,\gamma_*)$ and $\eta^*=\max(1,\gamma^*)$.
\end{itemize}
Finally, the global inner product over the virtual element edge space
$\VvhEdge(\P)$ is given by adding all the elemental contributions:
\begin{align*}
  \big[\Evh,\Fvh\big]_{\EDGE} := \sum_{\P\in\Th} \big[\Evh,\Fvh\big]_{\EDGE,\P}.
\end{align*}
In what follows, we shall use the following discrete norm:
\[
\norm{\cdot}{\VvhEdge}^2:= \big[\cdot,\cdot\big]_{\EDGE}.
\]

\subsection{The face space}
\label{subsection:face-space}

We define the virtual element face space on an element~$\P\in\Th$ as
\begin{align*}
  \VvhFace(\P) :=
  \bigg\{
  \Bvh
  & \in[\LS{2}(\P)]^3\,\mid\,
  \Bvh\cdot\nv_{\F}\in\PS{0}(\F)\,\forall\F\in \setFaceP,\nonumber\\
  &
  \begin{cases}
    \div \Bvh\in\PS{0}(\P),\nonumber\\
    \curl\Bvh\in\big[\PS{0}(\P)\big]^3
  \end{cases}
  \hspace{-4mm},\qquad
  \int_\P\Bvh\cdot(\xvP\times\qv) = 0 \ \forall\qv\in\big[\PS{0}(\P)\big]^3
  \bigg\},
\end{align*}
where $\xvP=\xv-\bvP$.

\medskip
Each element $\Bvh\in\VvhFace(\P)$ is uniquely determined by its
normal components on the elemental faces, i.e., the set of constant
values of $\Bs_{\F}=\big(\Bvh\cdot\nv_{\F}\big)_{\F\in\setFaceP}$.
We take these values as the degrees of freedom of $\Bvh$ and refer the
reader to~\cite{BeiraodaVeiga-Brezzi-Dassi-Marini-Russo:2018} for the
proof of their unisolvence.
A major consequence of such a space definition is that~$\PizvP{0}\Bvh$
and the divergence of~$\Bvh$ are computable from the degrees of
freedom of~$\Bvh$ for all~$\Bvh\in\VvhFace(\P)$.

\medskip
We define the global face space~$\VvhFace$ by an
$\Hs{}(\div)$-coupling of the degrees of freedom:
\[
  \VvhFace :=
  \Big\{
  \Bvh\in\HSzr{}(\div,\Omega)
  \,\mid\,
  \restrict{\Bvh}{\P}\in\VvhFace(\P)\,\,\forall\P\in\Th
  \Big\}.
\]
This definition includes the homogeneous normal boundary conditions
on~$\partial\Omega$.

\subsubsection{Virtual element inner product}

From~\cite{BeiraodaVeiga-Brezzi-Marini-Russo:2016,BeiraodaVeiga-Dassi-Manzini-Mascotto:2021},
we recall the discretization of the $\LS{2}$-inner product for the
local face virtual element space, which reads as
\begin{multline} 
  \big[\Bvh,\Cvh\big]_{\FACE,\P}
  :=
  \Big(\PizvP{0}\Bvh,\PizvP{0}\Cvh\Big)_{0,\P} +
  \SPFace( (\Is-\PizvP{0})\Bvh, (\Is-\PizvP{0})\Cvh )
  \\
  \quad \forall\Bvh,\,\Cvh\in\VvhFace(\P),
  \label{discrete:bf:face}
\end{multline}
where $\SPFace(\cdot,\cdot)$ can be any computable, symmetric bilinear form such that there exist two positive constants~$\widetilde{\gamma}_*$ and~$\widetilde{\gamma}^*$ independent of~$\hE$ satisfying
\begin{equation} 
  \widetilde{\gamma}_* \Vert\Bvh\Vert^2_{0,\P}
  \leq
  \SPFace(\Bvh,\Bvh)
  \leq 
  \widetilde{\gamma}^* \Vert\Bvh\Vert^2_{0,\P}
  \quad\quad \forall\Bvh\in\ker(\PizvP{0})\cap\VvhFace(\P).
  \label{stability:preliminary:face}
\end{equation}
An explicit stabilization satisfying~\eqref{stability:preliminary:face} was introduced in~\cite[formula~(4.8)]{BeiraodaVeiga-Brezzi-Dassi-Marini-Russo:2018}
analyzed in~\cite[Proposition~5.2]{BeiraodaVeiga-Mascotto:2021}, and reads
\[
  \SPFace(\Bvh,\Cvh) :=
  \hE\sum_{\F\in\setFaceP} \big(\nv_{\F}\cdot\Bvh,\nv_{\F}\cdot\Cvh\big)_{0,\F}
  \quad\quad \forall\Bvh,\,\Cvh\in\VvhFace(\P).
\]
The bilinear form~\eqref{discrete:bf:face} is computable on all
elements $\P\in\Th$ using the degrees of freedom of
$\Bvh,\Cvh\in\VvhFace(\P)$, and has the two crucial properties of consistency and
stability:
\begin{itemize}
\item \textbf{consistency}: for all $\qv\in\big[\PS{0}(\P)\big]^3$ and all $\Bvh\in\VvhFace(\P)$,
\[
\big[\qv,\Bvh\big]_{\FACE,\P} = (\qv,\Bvh)_{\P};
\]
\item \textbf{stability}: for all $\Bvh\in\VvhFace(\P)$:
  \begin{align}  
    \chi_*\Vert\Bvh\Vert^2_{0,\P} 
    \le
    \big[\Bvh,\Bvh\big]_{\FACE,\P}
    \le
    \chi^*\Vert\Bvh\Vert^2_{0,\P},
    \label{stability:face} 
  \end{align}
  with $\chi_*=\min(1,\widetilde{\gamma}_*)$ and
  $\chi^*=\max(1,\widetilde{\gamma}^*)$.
\end{itemize}
Finally, the inner product over the global virtual element space
$\VvhFace$ is given by summing all the elemental contributions:
\begin{align*}
  \big[\Bvh,\Cvh\big]_{\FACE} :=
  \sum_{\P\in\Th} \big[\Bvh,\Cvh\big]_{\FACE,\P}.
\end{align*}
In what follows, we shall use the following discrete norm:
\[
\norm{\cdot}{\VvhFace}^2:= \big[\cdot,\cdot\big]_{\FACE}.
\]

\subsection{Interpolations and exact sequence properties} \label{subsection:exact-sequences}

\setlength{\unitlength}{0.30mm}
\begin{figure}[b]
  \begin{center}
  \begin{picture}(400,100)

    \put(110, 50){$\HS{}(\curl,\Omega)$}
    \put(240, 50){$\HS{}(\div, \Omega)$}
    \put(170, 54){\vector(1,0){65}}
    \put(190, 59){$\curl$}

    \put(120,  0){$\VvhEdge$}
    \put(240,  0){$\VvhtFace$}
    \put(170,  3){\vector(1,0){60}}
    \put(190, 10){$\curl$}

    \put(125, 44){\vector(0,-1){30}}
    \put(248, 44){\vector(0,-1){30}}

    \put(131, 25){$\IntpEh$}
    \put(254, 25){$\IntpFh$}
    
  \end{picture}
  \end{center}
  \caption{Commuting diagram for the edge and face virtual element
    functions.}
  \label{fig:DeRham:primal:mesh:functions}
\end{figure}

As thoroughly discussed, e.g., in~\cite{BeiraodaVeiga-Brezzi-Dassi-Marini-Russo:2018},
the above spaces satisfy fundamental exact sequence properties.
Introduce an edge approximation operator $\IntpEh:\HS{}(\curl,\Omega)\to\VvhEdge$ as follows.
The edge approximant~$\IntpEh(\Ev) := \EvI\in\VvhEdge$ of a given vector field $\Ev\in\HS{}(\curl,\Omega)$ is defined as the solution to the variational problem
\begin{subequations}
  \label{eq:EvI:intp}
  \begin{align}
    \scalFace{\curl\EvI}{\curl\Fvh} &= \scal{\curl\Ev}{\curl\Fvh} && \hspace{-1.25cm}\forall\Fvh\in\VvhEdge, \label{eq:EvI:intp:A}\\[0.5em]
    \scalEdge{\EvI}{\nabla\rsh}     &= \scal{\Ev}{\nabla\rsh}     && \hspace{-1.25cm}\forall\rsh\in\VshNode, \label{eq:EvI:intp:B}
  \end{align}
\end{subequations}
where the spaces~$\VvhEdge$ and~$\VshNode$ are defined in~\eqref{global-edge-space} and~\eqref{global-nodal-space}, respectively.
The following approximation bound is valid~\cite[Proposition~5]{BeiraodaVeiga-Dassi-Manzini-Mascotto:2021}.
\begin{lem} \label{lemma:intp:Ev}
Let $\Ev\in\HS{}(\curl,\Omega)$ and consider the approximant~$\EvI\in\VvhEdge$ defined as in~\eqref{eq:EvI:intp}.
Then, there exists a real, positive constant~$\Cintp$ independent of~$\hh$ such that
\[
\norm{\Ev-\EvI}{\curlbold} 
\leq \Cintp \hh\big(\snorm{\Ev}{1,\Omega} + \norm{\curlbold \Ev}{0,\Omega} + \hh \snorm{\curlbold \Ev}{1,\Omega}\big).
\]
\end{lem}
Introduce a face approximation operator $\IntpFh:\HS{}(\div,\Omega)\to\VvhFace$ as follows.
The face approximant
\[
\IntpFh (\Bv):= \BvII \in \VvhtFace
:=\left\{ \Cvh \in \VvhFace \mid \div \Cvh =0  \right\}
\]
of a given vector field~$\Bv \in H(\div,\Omega)$ is defined as the solution to the variational problem
\begin{align} \label{eq:BvI:intp}
\scalFace{ \BvII }{ \Cvh } = \scal{ \Bv }{ \Cvh } \qquad\forall\Cvh \in \VvhtFace.
\end{align}
We have that~$\VvhtFace = \curl(\VvhEdge)$; 
see~\cite[equation~(4.35)]{BeiraodaVeiga-Brezzi-Dassi-Marini-Russo:2018}.

We recall the following approximation result~\cite[Lemma~7]{BeiraodaVeiga-Dassi-Manzini-Mascotto:2021}.
\begin{lem}  \label{lemma:intp:Bv}
Let~$\Bv \in [H^1(\Omega)]^3$ be a divergence-free magnetic field
and~$\BvII\in\VvhFace$ be its approximant defined as
in~\eqref{eq:BvI:intp}.  Then, there exists a real, positive constant
$\CI$ independent of~$\hh$ such that
\[
\norm{\Bv-\BvII}{} \leq \CI \hh \norm{\Bv}{1}.
\]
\end{lem}
Additionally, we have a commuting property involving the operators
defined in~\eqref{eq:EvI:intp} and~\eqref{eq:BvI:intp}, and the curl
operator, which we report in the next lemma, whose proof can be found
in~\cite[Proposition~6]{BeiraodaVeiga-Dassi-Manzini-Mascotto:2021}.
The interpolation diagram is shown in
Figure~\ref{fig:DeRham:primal:mesh:functions}.
\begin{lem}  \label{lem:commuting-diagram}
  Let~$\EvI \in \VvhEdge$ and~$(\curl\Ev)_{\INTPINTP} \in \VvhFace$ be
  the approximant of~$\Ev$ and~$\curl\Ev\in\HS{}(\div,\Omega)$ defined
  in~\eqref{eq:EvI:intp} and~\eqref{eq:BvI:intp}, respectively.  Then,
  the following commuting property is valid:
  \[
  \curl(\EvI) = (\curl\Ev)_{\INTPINTP}\quad\quad\forall\Ev\in\HS{}(\curl,\Omega).
  \]
\end{lem}

\subsection{The semi-discrete MHD model}
\label{subsection:semi-discrete-scheme}

First, note that $-(\jv\times\Bv,\vv)=(\jv,\vv\times\Bv)$ in
equation~\eqref{eq:MHD:weak:A}.
To discretize this term, we consider the elemental bilinear operator
$\chihP:\Wvh(\P)\times\VvhFace(\P)\to\VvhEdge(\P)$.
This operator denotes the discretization of the cross product
``$\vvh\times\Bvh$'' for $\vvh\in\Wvh(\P)$ and $\Bvh\in\VvhFace(\P)$
through the elemental edge function
$\chihP(\vvh,\Bvh)\in\VvhEdge(\P)$.
For the sake of notation, we also introduce the global operator
\begin{align*}
  \chih:\Wvh\times\VvhFace\to\Pi_{\P\in\Th}\VvhEdge(\P)
  \quad\textrm{such that}\quad
  \restrict{\chih(\vvh,\Bvh)}{\P}=\chihP(\restrict{\vvh}{\P},\restrict{\Bvh}{\P})
\end{align*}
for the global fields $\vvh\in\Wvh$ and $\Bvh\in\VvhFace$.
We emphasize that $\chih(\vvh,\Bvh)$ does not belong to $\VvhEdge$
since it is not $\HS{}(\curl,\Omega)$-conforming.
However, to simplify several equations below, with an abuse of
notation, we let
\begin{align*}
  \scalEdge{\Fvh}{\chih(\vvh,\Bvh)} := \sum_{\P\in\Th} \scalEdgeP{\Fvh}{\chih(\vvh,\Bvh)}.
\end{align*}
Using this bilinear operator, we define the ``broken'' virtual element edge function
\begin{align*}
  \jvh=\Evh+\chih(\uvh,\Bvh),
\end{align*}
which can be considered as the electric current density of the
discrete MHD model.
For the sake of simplicity, we take $\ss=1$ in~\eqref{eq:continuous:problem-strong:a} and~\eqref{eq:MHD:weak:A}.

Several choices for~$\chih$ are possible. Among them, we pick
\begin{equation} \label{choice:chih}
\chih(\vvh, \Bvh)_{|\E} := \PizP0 \vvh \times \PizP0 \Bvh.
\end{equation}

\medskip
The semi-discrete virtual element method for the MHD equations reads as:
\emph{For every~$t \in (0,T)$, find
$(\uvh,\psh,\Evh,\Bvh)\in\Wvh\times\Qsh\times\VvhEdge\times\VvhFace$
such that}\\
\begin{subequations}
  \label{eq:VEM:weak}
  \begin{align}
    \msh(\uvht,\vvh) + \REY^{-1}\ash(\uvh,\vvh) + \bs(\vvh,\psh) & + \cht(\uvh;\uvh,\vvh) + \scalEdge{\jvh}{\chih(\vvh,\Bvh)} \nonumber\\
    \qquad\qquad\qquad = (\fv,\Piz1\vvh) &\phantom{=~0} \qquad\forall\vvh\in\Wvh \subset\big[\HSzr{1}(\Omega)\big]^3,                      \label{eq:VEM:weak:A}\\[0.25em]
    \scalEdge{\jvh}{\Fvh}  - \REM^{-1}\scalFace{\Bvh}{\curl\Fvh}                                          &= 0 \qquad\forall\Fvh\in\VvhEdge\subset\HSzr{}(\curl,\Omega), \label{eq:VEM:weak:B}\\[0.25em]
    \scalFace{\Bvht}{\Cvh} + \scalFace{\curl\Evh}{\Cvh}                                                  &= 0 \qquad\forall\Cvh\in\VvhFace\subset\HSzr{}(\div,\Omega),  \label{eq:VEM:weak:C}\\[0.25em]
    \bs(\uvh,\qsh)                                                                                       &= 0 \qquad\forall\qsh\in\Qsh    \subset\LSzr{}(\Omega).  \label{eq:VEM:weak:D}
  \end{align}
\end{subequations}
We provide the semi-discrete scheme~\eqref{eq:VEM:weak} with the discrete initial conditions $\uvh(0)=\uv_{\hh,0}$ and $\Bvh(0)=\Bv_{\hh,0}$,
where~$\uv_{\hh,0}$ and~$\Bv_{\hh,0}$ are the degrees of freedom interpolants in~$\Wvh$ and~$\VvhFace$ of the continuous initial conditions~$\uv_0$ and~$\Bv_0$ in~\eqref{eq:initial:conditions}.
Provided that~$\Bv$ is sufficiently regular,
from~\cite[Proposition~$4.5$]{BeiraodaVeiga-Mascotto:2021}, the degrees of freedom interpolant of~$\Bv$ in the face space~$\VvhFace(\P)$ satisfies
\begin{equation} \label{interpolation:dofi-face}
\norm{\Bv(0)-\Bv_{\hh,0}}{0,\E} \lesssim \hP.
\end{equation}

\begin{rem} \label{remark:zero-divergence-discrete}
Recalling that~$\div \uvh \in \Qsh$ and~$\curlbold \Evh \in \VvhFace$,
the third and fourth equations provide
\begin{equation} \label{eq:discrete-solenoidal-constraints}
\Bvht + \curlbold \Evh = \mathbf 0, \qquad \div \uvh = 0 \qquad \text{in } \Omega.
\end{equation}
As a consequence, by setting the initial data such that $\div \Bv_{\hh,0}=0$,
the first equation in~\eqref{eq:discrete-solenoidal-constraints} implies $\div \Bvh=0$ at all times.
Thus, the proposed scheme satisfies both solenoidal constraints exactly.
\end{rem}

We conclude this section by presenting the following stability result
on the solution to the semi-discrete scheme~\eqref{eq:VEM:weak} and
briefly commenting on the existence of a (unique) solution to
\eqref{eq:VEM:weak}.
\begin{prop} \label{proposition:stability-discrete}
Let~$\uvh$, $\Evh$, and~$\Bvh$ be solutions to~\eqref{eq:VEM:weak}.
Then, the following stability estimates are valid:
\[
\begin{split}
\frac12 \frac{\text{d}}{\text{dt}} \msh(\uvh,\uvh)
 + \frac12 \REM^{-1} \frac{\text{d}}{\text{dt}} \scalFace{\Bvh}{\Bvh}
 + \REY^{-1} \ash(\uvh,\uvh) 
 + \scalEdge{\jvh}{\jvh} = (\fv, \Piz1 \uvh) \quad \forall t \in (0,T)
\end{split}
\]
and
\[
\begin{split}
& \max_{0\le t\le T} \left( \msh(\uvh,\uvh) + \REM^{-1} \scalFace{\Bvh}{\Bvh} \right)
+ \text{Re}^{-1} \int_0^T \ash(\uvh,\uvh)
+ 2 \int_0^T \scalEdge{\jvh}{\jvh}\\
& \quad
\le \msh(\uv_{\hh,0},\uv_{\hh,0}) + \REM^{-1}  \scalFace{\Bv_{\hh,0}}{\Bv_{\hh,0}}
   + \alpha_*^{-1} \CP \REY \int_0^T \norm{\fv}{0}.
\end{split}
\]
\end{prop}
\begin{proof}
The first estimate follows picking $\vvh=\uvh$
in~\eqref{eq:VEM:weak:A}, $\Fvh=\Evh$ in~\eqref{eq:VEM:weak:B},
and~$\Cvh=\Bvh$ in~\eqref{eq:VEM:weak:C},
multiplying~\eqref{eq:VEM:weak:B} by~$\REM$, and summing up the three
resulting equations.

As for the second one, we observe that
\[
\begin{split}
(\fv,\Piz1 \uvh)
& \le \frac12 \alpha_*^{-1} \CP \REY \norm{\fv}{0}^2 
   + \frac12 \alpha_* \CP^{-1} \REY^{-1} \norm{\Piz1 \uvh}{} ^2
\le \frac12 \alpha_*^{-1} \CP \REY \norm{\fv}{0}^2 + \frac12 \alpha_* \REY^{-1} \snorm{\uvh}{1}^2\\
& \le \frac12\alpha_*^{-1} \CP \REY \norm{\fv}{0}^2 + \frac12\REY^{-1} \ash(\uvh,\uvh).
\end{split}
\]
Next, we absorb the second term on the right-hand side in the
left-hand side of the first bound.  The assertion follows by
integrating in time the first bound between~$0$ and any time~$0<t\le
T$, multiplying by two both sides, and taking the maximum over the
time~$t$.
\end{proof}

Existence and uniqueness of a solution to the semi-discrete problem~\eqref{eq:VEM:weak}
follows from the standard ODE theory in finite dimensions.
Here, we only sketch the proof.
First, we eliminate the pressure variable and
equation~$\eqref{eq:VEM:weak:D}$ by restricting the (test and trial)
velocity space to divergence-free functions.
Next, \eqref{eq:VEM:weak:B} at the initial time~$t=0$ together with
the definition of~$\jvh$ allows us to define an initial condition
$\Ev_{h,0}$ for the electric field, depending on the initial
conditions~$\uv_{\hh,0}$ and~$\Bv_{\hh,0}$ of the velocity and
magnetic fields.
Then, we can substitute~$\eqref{eq:VEM:weak:B}$ with the same equation
derived in time:
\[
\scalEdge{\jvht}{\Fvh}  - \REM^{-1}\scalFace{\Bvht}{\curl\Fvh} = 0 \qquad\forall\Fvh\in\VvhEdge.
\]
Recalling the definition of~$\jvh$, deriving in time once, and applying~\eqref{eq:VEM:weak:C}, we
obtain
\begin{equation}\label{new:L:timeeq}
  \scalEdge{{\Ev}_{h,t}}{\Fvh} +  \REM^{-1} \scalFace{\curl\Evh}{\curl\Fvh} 
  + \scalEdge{\partial_t \chih(\vvh, \Bvh)}{\Fvh}
  = 0 
\qquad \forall \Fvh \in \VvhEdge.
\end{equation}
Substituting~\eqref{eq:VEM:weak:B} with~\eqref{new:L:timeeq} and recalling that~${\bf E}_{h,0}$ is known
yield an equivalent system,
which is a first order Cauchy problem
on the finite-dimensional variable~$(\uvh,\Bvh,\Evh)$:
for all~$\vvh$ in~$\Wvh$, $\Fvh$ in~$\VvhEdge$, and~$\Cvh$ in~$\VvhFace$,
\begin{equation} \label{rewritten-system}
\begin{cases}
\msh(\uvht,\vvh) + \REY^{-1}\ash(\uvh,\vvh) + \cht(\uvh;\uvh,\vvh) + \scalEdge{\jvh}{\chih(\vvh,\Bvh)}
																																							&= (\fv,\Piz1\vvh) ,           \\[0.25em]
 \scalEdge{{\Ev}_{h,t}}{\Fvh} +  \REM^{-1} \scalFace{\curl\Evh}{\curl\Fvh}   + \scalEdge{\partial_t \chih(\vvh, \Bvh)}{\Fvh}  &= 0 , \\[0.25em]
\scalFace{\Bvht}{\Cvh} + \scalFace{\curl\Evh}{\Cvh}  																							&= 0 .
\end{cases}
\end{equation}
The term
\begin{align*}
\scalEdge{\partial_t \chih(\vvh, \Bvh)}{\Fvh} = 
\scalEdge{\PizP0 \uvht \times \PizP0 \Bvh}{\Fvh}
+ \scalEdge{\PizP0 \uvh \times \PizP0 \Bvht}{\Fvh}
\end{align*}
can be interpreted as a quartic function of~$(\uvh,\Bvh,\Evh)$.
In fact, using the algebraic form of the first and third equation in~\eqref{rewritten-system},
and inverting the corresponding ``mass'' matrices,
give that~$\uvht$ can be interpreted as a cubic function of~$(\uvh,\Bvh,\Evh)$
and~$\Bvht$ as a linear function of~$\Evh$.
This implies that~$\PizP0 \uvht \times \PizP0 \Bvh$ and~$\PizP0 \uvh \times \PizP0 \Bvht$
can be interpreted as a quartic and quadratic functions of~$(\uvh,\Bvh,\Evh)$.

Consequently, \eqref{rewritten-system} is a first order Cauchy problem in~$(\uvh, \Bvh, \Evh)$ with quartic right-hand side.
Proposition~\ref{proposition:stability-discrete} asserts that the solutions to the system (if they exist) cannot blow up in finite time.
Thus, the nonlinearity on the right-hand side is uniformly Lipschitz.
For finite times, existence and uniqueness follow from the standard ODE theory in finite dimension;
see, e.g., \cite[Section~$9.2$]{Layton-book}.


\section{Convergence analysis of the virtual element approximation}
\label{section:VEM:analysis}

Introduce the ``discrete errors''
\begin{align}
  \evhu:=\uvh-\uvI,\qquad
  \evhE:=\Evh-\EvI,\qquad
  \evhB:=\Bvh-\BvII,
  \label{eq:convenient:notation:1}
\end{align}
and the virtual element edge fields
\begin{align}
  \jvI &= \EvI + \chih(\uvI, \Bvh)
  \label{eq:convenient:notation:1.5}
\end{align}
and
\begin{equation} \label{eq:convenient:notation:2}
\evhJ
= \jvh - \jvI
= \big(\Evh + \chih(\uvh, \Bvh)\big) - \big(\EvI + \chih(\uvI, \Bvh)\big) 
= \evhE + \chih(\evhu,\Bvh).
\end{equation}
These quantities are bounded as in the following theorem, whose proof can be found in Section~\ref{subsection:proof-Theorem-5.1}.
\begin{thm} \label{thm:convergence:1}
Let $(\uv(t),\ps(t)$, $\Ev(t),\Bv(t))$ in
$\big[\HSzr{1}(\Omega)\big]^3\times
\LSzr{2}(\Omega)\times\HSzr{}(\curl,\Omega)\times\HSzr{}(\div,\Omega)$
for almost every $t\in(0,T)$ be the solution to the variational
formulation~\eqref{eq:continuous:problem-weak} under the assumptions
in Section~\ref{section:continuous-problem}.  Furthermore, assume
that the following regularity properties
\begin{align}
  \uv \in [W^{2,\infty}(\Omega)]^3, \qquad
  \uvt \in [H^{1}(\Omega)]^3, \qquad
  \Bv \in [H^1(\Omega) \cap L^{\infty}(\Omega)]^3 \quad\textrm{and}\quad
  \jv \in [L^{\infty}(\Omega)]^3
  \label{eq:Theorem1:regularity:assumptions}
\end{align}
hold uniformly in time.

Let $(\uvh(t),\psh(t),\Evh(t),\Bvh(t))$ in $\Wvh\times\Qsh\times\VvhEdge\times\VvhFace$ for every~$t \in (0,T)$ be the solution to the virtual element method~\eqref{eq:VEM:weak}.
Then, a positive constant~$\Cs$ exists that is independent of $\hh$ such that
\[
\norm{\evhu(t)}{} + \norm{\evhB(t)}{} + \left( \int_0^t \norm{\evhJ(s)}{}^2 \right)^{\frac12}
\leq \Cs\big( \norm{\evhu(0)}{} + \norm{\evhB(0)}{} + \hh \big) \qquad \text{for every } t\in (0,T],
\]
where $\evhu$, $\evhB$, and $\evhJ$ are defined in~\eqref{eq:convenient:notation:1}-\eqref{eq:convenient:notation:2}.
The constant~$\Cs$ depends on the parameters of the discretization,
the final time~$T$, and the regularity of the MHD solution fields $\uv(t),\Ev(t),\Bv(t)$, and $\Bvh(t)$.
\end{thm}
In view of Theorem~\ref{thm:convergence:1}, we have the following convergence result.
\begin{thm}\label{thm:convergence:2}
Let $(\uv(t),\ps(t)$, $\Ev(t),\Bv(t))$ in $\big[\HSzr{1}(\Omega)\big]^3\times \LSzr{2}(\Omega)\times\HSzr{}(\curl,\Omega)\times\HSzr{}(\div,\Omega)$ for almost every~$t\in(0,T)$
be the solution to the variational formulation~\eqref{eq:continuous:problem-weak} under the assumptions in Section~\ref{section:continuous-problem}.
Let the regularity assumptions of Theorem~\ref{thm:convergence:1} be valid
and assume that the initial conditions~$\uv_0$, $\Bv_0$, and~$\Ev_0$ of the fields~$\uv$, $\Bv$, and~$\Ev$ belong to~$[H^1(\Omega)]^3$.
Let $(\uvh(t),\psh(t),\Evh(t),\Bvh(t))$ in $\Wvh\times\Qsh\times\VvhEdge\times\VvhFace$ for every~$t\in(0,T)$ be the solution to the virtual element method~\eqref{eq:VEM:weak}.
Then, a positive constant~$\Cs$ exists, independent of~$\hh$, such that
\begin{equation} \label{cane}
\norm{\uv(t)-\uvh(t)}{} + \norm{\Bv(t)-\Bvh(t)}{} + \left(\int_0^t \norm{\Ev(s)-\Evh(s)}{}^2 \right)^{\frac12} \le \Cs \hh
\end{equation}
for almost every $t\in(0,T]$.
The constant~$\Cs$ depends on the parameters of the discretization,
the final time~$T$, and the regularity of the MHD solution fields
$\uv(t),\Ev(t),\Bv(t)$, and $\Bvh(t)$.
\end{thm}
\begin{proof}
We add and subtract~$\uvI$ (see Lemma~\ref{lemma:uvI:intp}), $\EvII$ (see~\eqref{eq:EvI:intp}), and~$\BvII$ (see~\eqref{eq:BvI:intp}) to the left-hand side of~\eqref{cane} and get
\[
\begin{split}
& \norm{\uv(t) - \uvh(t)}{} + \norm{\Bv(t) - \Bvh(t)}{} + \left( \int_0^t \norm{\Ev(s) - \Evh(s)}{}^2 \right)^{\frac12}\\
& \le \left[ \norm{\uv(t) - \uvI(t)}{} + \norm{\Bv(t) - \BvII(t)}{} + \left(\int_0^t \norm{\Ev(s) - \EvII(s)}{}^2 \right)^{\frac12} \right] \\
& \quad     + \left[ \norm{\evhu(t)}{} + \norm{\evhB(t)}{} + \left( \int_0^t \norm{\evhE(s)}{}^2 \right)^{\frac12} \right] := S_1+S_2.
\end{split}
\]
An upper bound on the term~$S_1$ can be shown using the interpolation properties
of Lemmas~\ref{lemma:uvI:intp}, \ref{lemma:intp:Bv}, and~\ref{lemma:intp:Ev}.

Observe that~$\evhu(0)=0$.
Furthermore, the triangle inequality, \eqref{interpolation:dofi-face}, and Lemma~\ref{lemma:intp:Bv} give
\[
\norm{\evhB(0)}{}
\le \norm{\Bv(0) - \Bv_{\hh,0}}{} + \norm{\Bv(0)-\BvII(0)}{}
\lesssim \hh.
\]
With this at hand, Theorem~\ref{thm:convergence:1} allows us to show a direct bound on the terms involving the velocity and magnetic fields appearing in~$S_2$.
For the term involving the electric field we proceed as follows:
using the triangle inequality, the definition of~$\chih$ in~\eqref{choice:chih},
the bound on~$\norm{\evhu}{}$ from~\eqref{eq:Theorem1:regularity:assumptions},
the stability estimates on the discrete face element bilinear forms~\eqref{stability:face},
and the bound on~$\scalFace{\Bvh}{\Bvh}$ in Proposition~\ref{proposition:stability-discrete},
we arrive at
\[
\begin{split}
& \left(\int_0^t \norm{\evhE(s)}{}^2 \right)^{\frac12}
\lesssim
\left( \int_0^t \left(\norm{\evhJ(s)}{}^2 + \norm{\chih(\evhu(s), \Bvh(s))}{}^2 \right) \right)^{\frac12}\\
& \quad \le
\hh +  \left(\int_0^t \norm{\evhu(s)}{}^2 \norm{\Bvh(s)}{}^2 \right)^{\frac12}
  \lesssim
\hh + \norm{\evhu}{\LS{\infty}(0,t; L^2(\Omega))} \norm{\Bvh}{\LS{\infty}(0,t;\VvhFace)} 
\lesssim  \hh.
\end{split}
\]
The final hidden constant depends on the square root of the instant time~$t$.
\end{proof}

\subsection{Proof of Theorem~\ref{thm:convergence:1}} \label{subsection:proof-Theorem-5.1}
\begin{proof}
Throughout, for presentation's sake, we assume that~$\REM=1$ and~$\REY=1$. The general assertion can be proved by means of minor modifications
but note that the final constant in Theorem~\ref{thm:convergence:1} will depend on~$\REM$ and~$\REY$; treating convection-dominated cases is beyond the scope of this work.\medskip

First, we add and subtract~$\EvI$ and~$\chih(\uvI,\Bvh)$ to~$\jvh=\Evh+\chih(\uvh,\Bvh)$ so that
  \begin{align}
    \jvh
    = \Evh + \chih(\uvh,\Bvh)
    = (\Evh-\EvI) + \chih(\uvh-\uvI,\Bvh) + \EvI + \chih(\uvI,\Bvh),
    \label{eq:jvh:00}
  \end{align}
where~$\EvI$ and~$\uvI$ are defined in~\eqref{eq:EvI:intp} and Section~\ref{subsubsection:interpolation-velocity} (see also Lemma~\ref{lemma:uvI:intp}), respectively.
  We substitute~\eqref{eq:jvh:00} in~\eqref{eq:VEM:weak:A}, add and
  subtract $\uvI$, and find that
  \begin{align}
    &\msh\big( (\uvh-\uvI)_t,\vvh \big) + \ash(\uvh-\uvI,\vvh) + \bs(\vvh,\psh)
    \nonumber\\[0.5em]
    &\hspace{6cm}
    + \scalEdge{\Evh-\EvI + \chih(\uvh-\uvI)}{ \chih(\vvh,\Bvh) }
    \nonumber\\[0.5em]
    &\qquad
    = (\fv,\Piz1 \vvh)
    - \msh( \uvIt,\vvh \big) - \ash(\uvI,\vvh)
    - \cht(\uvh;\uvh,\vvh)
    \nonumber\\[0.5em]
    &\qquad\phantom{=}
    -\scalEdge{ \EvI + \chih(\uvI,\Bvh) }{ \chih(\vvh,\Bvh) }.
    \label{eq:proof:00}
  \end{align}
Similarly, we use~\eqref{eq:jvh:00} in~\eqref{eq:VEM:weak:B}, add and subtract~$\BvII$ defined in~\eqref{eq:BvI:intp}, and find that
\begin{multline}    \label{eq:proof:10}
\scalEdge{ \Evh-\EvI + \chih(\uvh-\uvI,\Bvh) }{ \Fvh } - \scalFace{ \Bvh-\BvII }{ \curl\Fvh } \\[0.5em]
= - \scalEdge{\EvI + \chih(\uvI,\Bvh)}{ \Fvh } + \scalFace{ \BvII }{ \curl\Fvh }.
\end{multline}
Then, we use Lemma~\ref{lem:commuting-diagram}, \eqref{eq:BvI:intp} twice and~\eqref{eq:MHD:weak:C}, and find that,
for all~$\Cvh \in \VvhtFace$,
  \begin{align*}
    \scalFace{\curl\EvI}{\Cvh}
    = \scalFace{(\curl\Ev)_{\INTPINTP}}{\Cvh}
    =  ( \curl\Ev, \Cvh)
    = -( \Bvt,     \Cvh)
    = -\scalFace{(\Bvt)_{\INTPINTP}}{\Cvh}.
  \end{align*}
  Since the derivation in time and the interpolation commute, e.g., $(\Bvt)_{\INTPINTP}=\BvIIt$, we find
  \begin{align*}
    \scalFace{\BvIIt}{\Cvh} + \scalFace{\curl\EvI}{\Cvh} = 0.
  \end{align*}
Finally, we subtract this equation to~\eqref{eq:VEM:weak:C} and obtain,
for all~$\Cvh \in \VvhtFace$,
\begin{align}
  \scalFace{(\Bvh-\BvII)_t}{\Cvh} + \scalFace{\curl(\Evh-\EvI)}{\Cvh}
  = 0.
  \label{eq:proof:15}
\end{align}

\medskip
\noindent
Next, we take $\vvh=\evhu$ in~\eqref{eq:proof:00}, $\Fvh=\evhE$
in~\eqref{eq:proof:10}, and $\Cvh=\evhB$ in~\eqref{eq:proof:15}.
This choice of~$\Cvh$ is admissible since~$\div\evhB=0$ for every $t\geq0$ thanks to Remark~\ref{remark:zero-divergence-discrete} and the fact that~$\BvII$
has zero divergence by definition.
We further observe~$\bs(\evhu,\psh)=0$ since
\begin{align*}
    \evhu\in\Zvh:=\big\{\vvh\in\Wvh\,\mid\,(\div\vvh,\qsh)=0\quad\forall\qsh\in\Qsh\big\};
\end{align*}
see also~\eqref{uI-div}.
Equations~\eqref{eq:proof:00}, \eqref{eq:proof:10}, and~\eqref{eq:proof:15} thus become
  \begin{align}
    &\msh\big( \evhut, \evhu \big)
    + \ash(\evhu,\evhu)
    + \scalEdge{ \evhE + \chih(\evhu,\Bvh) }{ \chih(\evhu,\Bvh) }
    \nonumber\\[0.5em] &\qquad
    = (\fv,\Piz1 \evhu)
    - \msh( \uvIt, \evhu \big)
    - \ash(\uvI,\evhu)
    - \cht(\uvh;\uvh,\evhu)
    \nonumber\\[0.5em] &\qquad\phantom{=}
    -\scalEdge{ \EvI + \chih(\uvI,\Bvh) }{ \chih(\evhu,\Bvh) };
    \label{eq:proof:20a}
    \\[0.5em] &
    \scalEdge{ \evhE + \chih(\evhu,\Bvh) }{ \evhE }
    - \scalFace{ \evhB }{ \curl\evhE }
    = - \scalEdge{ \big( \EvI + \chih(\uvI,\Bvh) \big) }{ \evhE }
    \nonumber\\[0.5em] &
    \phantom{
      \scalEdge{ \evhE + \chih(\evhu,\Bvh) }{ \evhE }
      - \scalFace{ \evhB }{ \curl\evhE } =
    }
    + \scalFace{ \BvII }{ \curl\evhE };
    \label{eq:proof:20b}
    \\[0.5em]
    &\scalFace{\evhBt}{\evhB} + \scalFace{\curl\evhE}{\evhB} = 0.
    \label{eq:proof:20c}
  \end{align}
  We add~\eqref{eq:proof:20c} to~\eqref{eq:proof:20b}, so that
  \begin{align*}
    \scalEdge{ \evhE + \chih(\evhu,\Bvh) }{ \evhE }
    + \scalFace{\evhBt}{\evhB}
    &= - \scalEdge{ \big( \EvI + \chih(\uvI,\Bvh) \big) }{ \evhE }
    \nonumber\\[0.5em] &\phantom{=\,}
    + \scalFace{ \BvII }{ \curl\evhE }.
  \end{align*}
Then, we add this equation to~\eqref{eq:proof:20a} and write the resulting equation as
  \begin{align}
    \LHS
    &:=
    \msh(\evhut,\evhu\big)
    + \ash(\evhu,\evhu)
    + \scalFace{\evhBt}{\evhB}
    \nonumber
    + \scalEdge{\evhE+\chih(\evhu,\Bvh)}{\evhE + \chih(\evhu,\Bvh)}
    \nonumber\\[0.5em] &
    = (\fv,\Piz1 \evhu)
    -\msh( \uvIt,\evhu \big)
    - \ash(\uvI,\evhu)
    - \cht(\uvh;\uvh,\evhu)
    \nonumber\\[0.5em] &
    \phantom{:=}
    -\scalEdge{\EvI+\chih(\uvI,\Bvh)}{\chih(\evhu, \Bvh)}
    \nonumber\\[0.5em] &
    \phantom{:=}
    - \scalEdge{\big(\EvI+\chih(\uvI,\Bvh)}{\evhE}
    + \scalFace{\BvII}{\curl\evhE}
    =: \RHS.
    \label{eq:LHS=RHS}
  \end{align}
We reformulate $\LHS$ and $\RHS$ in~\eqref{eq:LHS=RHS} as
\begin{equation} \label{LHS-RHS}
\begin{split}
\LHS
&= \frac{1}{2}\frac{d}{dt}\msh(\evhu,\evhu\big)
      + \frac{1}{2}  \frac{d}{dt}\scalFace{\evhB}{\evhB}
      + \ash(\evhu,\evhu) + \scalEdge{\evhJ}{\evhJ},\\[0.5em]
\RHS
&=  (\fv,\Piz1 \evhu)
      - \msh( \uvIt,\evhu \big) - \ash(\uvI,\evhu)
      - \cht(\uvh;\uvh,\evhu) -\scalEdge{\jvI}{\evhJ}\\[0.5em]
&\phantom{=} + \scalFace{\BvII}{\curl\evhE}.
\end{split}
\end{equation}
Next, we observe that
\begin{itemize}
\item equation~\eqref{eq:MHD:weak:A} with $\vv=\evhu \in \big[\HSzr{1}(\Omega)\big]^3$ can be rewritten as
\begin{align}  \label{eq:proof:30}
    (\fv,\evhu) = (\uvt,\evhu) + \as(\uv,\evhu) + \cst(\uv;\uv,\evhu) + (\jv,\evhu\times\Bv),
\end{align}
since $\bs(\evhu,\ps)=0$ as $\div\evhu=0$ and $\cst(\uv;\uv,\evhu)=\cs(\uv;\uv,\evhu)$ since~$\div\uv=0$ and $\uv\in\big[\HSzr{1}(\Omega)\big]$;
\item using the definition of~$\BvII$ in~\eqref{eq:BvI:intp}, and~\eqref{eq:MHD:weak:B} with    $\Fv=\evhE\in\HSzr{}(\curl,\Omega)$ yields
\begin{align} \label{eq:proof:35}
\scalFace{\BvII}{\curl\evhE} = (\Bv , \curl\evhE) = (\jv,\evhE).
\end{align}
  \end{itemize}
We substitute~\eqref{eq:proof:30} and~\eqref{eq:proof:35} in~$\RHS$ defined in~\eqref{LHS-RHS}, add and subtract $(\jv,\evhu\times\Bvh)$, 
use~\eqref{eq:convenient:notation:1}, and finally obtain
\begin{align}
\RHS
&:= (\fv,\Piz1 \evhu) - \msh( \uvIt,\evhu \big) - \ash(\uvI,\evhu)
    - \cht(\uvh;\uvh,\evhu)- \scalEdge{\jvI}{\evhJ} + (\jv,\evhE)    \nonumber\\[0.5em]
&= \Big[ (\uvt,\evhu) - \msh( \uvIt,\evhu \big) \Big]
    + \Big[ \as(\uv,\evhu) - \ash(\uvI,\evhu) \Big]
    + \Big[ \cst(\uv;\uv,\evhu) - \cht(\uvh;\uvh,\evhu) \Big] \nonumber\\[0.5em]
&\phantom{:=} + (\jv,\evhu\times\Bv) -\scalEdge{\jvI}{\evhJ} + (\jv,\evhE)
+ (\fv,\Piz1 \evhu) - (\fv, \evhu)  \nonumber\\[0.5em]
&= \Big[ (\uvt,\evhu) - \msh( \uvIt,\evhu \big) \Big]
   + \Big[ \cst(\uv;\uv,\evhu) - \cht(\uvh;\uvh,\evhu) \Big]
   + (\jv,\evhE+\evhu\times\Bv)\nonumber\\[0.5em]
& \phantom{:=} - \scalEdge{\jvI}{\evhJ} + \Big[ \as(\uv,\evhu) - \ash(\uvI,\evhu) \Big]
+ \Big[(\fv,\Piz1 \evhu) - (\fv, \evhu)\Big ] \nonumber\\[0.5em]
&= \Big[ (\uvt,\evhu) - \msh( \uvIt,\evhu \big) \Big]
+ \Big[ \cst(\uv;\uv,\evhu) - \cht(\uvh;\uvh,\evhu) \Big]
+ \Big[ \big(\jv,\evhu\times(\Bv-\Bvh)\big) \Big] \nonumber\\[0.5em]&\phantom{=}
+ \Big[ \big(\jv,\evhE+\evhu\times\Bvh\big) - \scalEdge{\jvI}{\evhJ} \Big]  + \Big[ \as(\uv,\evhu) - \ash(\uvI,\evhu) \Big]
+ \Big[(\fv,\Piz1 \evhu) - (\fv, \evhu)\Big ] \nonumber\\[0.5em]
&=: \TERM{T}{1} + \TERM{T}{2} + \TERM{T}{3} + \TERM{T}{4}+ \TERM{T}{5} + \TERM{T}{6}.
\label{eq:RHS:T-terms}
\end{align}
We estimate the six terms~$\TERM{T}{j}$, $j=1,\dots,6$, separately,
and we eventually apply Gr\" onwall's inequality to derive the assertion of the theorem.
  
\paragraph*{Estimate of $\TERM{T}{1}$}
We split $\TERM{T}{1}$ into a summation on the elemental contributions; 
note that $\uvIt=\uvtI$;
use the consistency property~\eqref{eq:velocity:L2:consistency} with the piecewise linear discontinuous polynomial approximation~$\uv_{t,\pi}$ of $\uvt$;
use the stability property~\eqref{eq:velocity:L2:stability}, the Cauchy-Schwarz inequality, and the Young inequality, and obtain:
\begin{align*}
\TERM{T}{1}
& = \sum_{\P\in\Th} \ABS{ \big(\uvt,\evhu\big)_{\P} -\mshP(\uvIt,\evhu) }
 = \sum_{\P\in\Th} \ABS{ \big(\uvt-\uv_{t,\pi},\evhu\big)_{\P} -\mshP(\uvtI-\uv_{t,\pi},\evhu) } \nonumber\\[0.5em] &
 \leq \sum_{\P\in\Th} \big(\norm{\uvt-\uv_{t,\pi}}{\P} + \mu^*\norm{\uvtI-\uv_{t,\pi}}{\P}\big)\,\norm{\evhu}{\P}
 \leq \big(\norm{\uvt-\uv_{t,\pi}}{} + \mu^*\norm{\uvtI-\uv_{t,\pi}}{}\big)\,\norm{\evhu}{} \nonumber\\[0.5em]
& \leq \big((1+\mu^*)\norm{\uvt-\uv_{t,\pi}}{} + \mu^*\norm{\uvt - \uvtI}{}\big)\,\norm{\evhu}{} \nonumber\\[0.5em] &
\leq \frac{1}{2\varepsilon}\Big(\big(1+\mu^*\big)^2\norm{\uvt-\uv_{t,\pi}}{}^2 + (\mu^*)^2\norm{\uvt-\uvtI}{}^2\Big) + \varepsilon\norm{\evhu}{}^2 \nonumber\\[0.5em] &
\leq \hh^2\,\Cnst{\TERM{T}{1}}(\varepsilon) + \varepsilon\norm{\evhu}{}^2,
\end{align*}
with $\Cnst{\TERM{T}{1}}(\varepsilon) := \big(\Cpi^2(1+\mu^*)^2 + \Cintp^2(\mu^*)^2 \big)\snorm{\uvt}{1}^2\slash{2\varepsilon}$.

\paragraph*{Estimate of $\TERM{T}{2}$}
We add and subtract $ \cht(\uv;\uv,\evhu)$, use the triangle inequality, and write
\begin{align*}
\TERM{T}{2}
&=   \ABS{\cst(\uv;\uv,\evhu) - \cht(\uvh;\uvh,\evhu)}
   \leq \ABS{\cst(\uv;\uv,\evhu) - \cht(\uv;\uv,\evhu)} + \ABS{\cht(\uv;\uv,\evhu) - \cht(\uvh;\uvh,\evhu)}
   \nonumber\\[0.5em]
&=: \TERM{T}{2,1} + \TERM{T}{2,2}.
  \end{align*}
We control $\TERM{T}{2,1}$ by using Lemma~\ref{lemma:property-trilinear2} and the Young inequality:
\begin{align*}
\TERM{T}{2,1}
\leq \CC \hh  \norm{\uv}{2}^2 \, \snorm{\evhu}{1}
\leq \frac{\CC^2}{4\varepsilon} \hh^2\norm{\uv}{2}^4 + \varepsilon\snorm{\evhu}{1}^2,
\end{align*}
where the constant~$\CC$ is independent of~$\hh$.
Similarly, we control $\TERM{T}{2,2}$ using Lemma~\ref{lemma:property-trilinear3}:
\begin{align*}
\TERM{T}{2,2}
& \leq  \varepsilon\snorm{\evhu}{1}^2
+ \Big(1+\frac{1}{\varepsilon}\Big) \Cs_1 R_1(\uv) \norm{\evhu}{}^2
 + \Big(1+\frac{1}{\varepsilon}\Big)\Cs_2 R_2(\uv) \hh^2,
  \end{align*}
where the constants $\Cs_1$ and $\Cs_2$ are independent of $\hh$.
We end up with
\[
\TERM{T}{2}
\leq \hh^2\,C_{\TERM{T}{2}}(\varepsilon)
+ \Big(1+\frac{1}{\varepsilon}\Big)\Cs_1 R_1(\uv) \norm{\evhu}{}^2 + 2\varepsilon\snorm{\evhu}{1}^2,
\]  
where
\[
C_{\TERM{T}{2}}(\varepsilon)
:= \frac{\CC^2}{4\varepsilon} \norm{\uv}{2}^4  + \Big(1+\frac{1}{\varepsilon}\Big)\Cs_2 R_2(\uv).
\]
\paragraph*{Estimate of $\TERM{T}{3}$}
By adding and subtracting~$\BvII$ defined in~\eqref{eq:BvI:intp}, we obtain
\begin{align*}
    \TERM{T}{3} &
    = \big(\jv,\evhu\times(\Bv-\Bvh)\big)
    = \big(\jv,\evhu\times(\Bv-\BvII)\big) + \big(\jv,\evhu\times(\BvII-\Bvh)\big)
    =: \TERM{T}{3,1} + \TERM{T}{3,2}.
\end{align*}
We estimate the two terms~$\TERM{T}{3,1}$ and~$\TERM{T}{3,2}$ separately.
Notably, we write
\begin{align*}
\ABS{\TERM{T}{3,1}} 
&\leq \norm{\jv}{\LS{\infty}}\,\norm{\evhu}{}\,\norm{\Bv-\BvII}{}
\leq \frac{\norm{\jv}{\LS{\infty}}}{2} \,\bigg(\norm{\evhu}{}^2 
    + \norm{\Bv-\BvII}{}^2\bigg) \\[0.5em]
& \leq \frac{\norm{\jv}{\LS{\infty}}}{2}\,\bigg(\norm{\evhu}{}^2 
  + \CII^2 \hh^2\snorm{\Bv}{1}^2\bigg)
\end{align*}
and
\begin{align*}
    \ABS{\TERM{T}{3,2}}
    \leq \norm{\jv}{\LS{\infty}}\,\norm{\evhu}{}\,\norm{\BvII-\Bvh}{}
    \leq \norm{\jv}{\LS{\infty}}\,\norm{\evhu}{}\,\norm{\evhB}{}
    \leq \frac{1}{2}\norm{\jv}{\LS{\infty}}\,\big(\norm{\evhu}{}^2+\norm{\evhB}{}^2\big).
\end{align*}
Collecting the two above estimates, we find that
\begin{align*}
\TERM{T}{3}
\leq \hh^2\,C_{\TERM{T}{3}}
     + \norm{\jv}{\LS{\infty}}\,\norm{\evhu}{}^2
     + \frac{1}{2}\norm{\jv}{\LS{\infty}}  \norm{\evhB}{}^2 ,
\end{align*}
where
\[
C_{\TERM{T}{3}}
:= \frac{\CII^2}{2}\,\norm{\jv}{\LS{\infty}}\snorm{\Bv}{1}^2.
\]

\paragraph*{Estimate of $\TERM{T}{4}$}
Recalling that $\evhJ=\evhE+\evhu\times\Bvh$, we use (by now) standard manipulations based on the consistency and stability properties~\eqref{consistency:edge} and~\eqref{stability:edge}, and get
\begin{align}
\TERM{T}{4}
& = \big(\jv,\evhJ\big) - \scalEdge{\jvI}{\evhJ}
  = \sum_{\P\in\Th} \Big( \big(\jv,\evhJ\big)_{\P} - \scalEdgeP{\jvI}{\evhJ} \Big) \nonumber\\[0.5em]
& = \sum_{\P\in\Th} \Big( \big(\jv-\jv_{\pi},\evhJ\big)_{\P} - \scalEdgeP{\jvI-\jv_{\pi}}{\evhJ} \Big) \nonumber
\leq\sum_{\P\in\Th} \Big( \norm{\jv-\jv_{\pi}}{0,\P} + \eta^* \norm{\jvI-\jv_{\pi}}{0,\P} \Big)\norm{\evhJ}{0,\P} \nonumber\\[0.5em]
& \leq \sum_{\P\in\Th} \Big( \big(1+\eta^*\big)\norm{\jv-\jv_{\pi}}{0,\P} + \eta^*\norm{\jv-\jvI}{0,\P} \Big)\norm{\evhJ}{0,\P} \nonumber
 \leq \Big( \big(1+\eta^*\big)\norm{\jv-\jv_{\pi}}{} + \eta^*\norm{\jv-\jvI}{} \Big)\norm{\evhJ}{} \nonumber\\[0.5em]
& \leq \frac{1}{2\varepsilon} \Big( \big(1+\eta^*\big)^2\norm{\jv-\jv_{\pi}}{}^2 + (\eta^*)^2\norm{\jv-\jvI}{}^2 \Big) + \varepsilon\norm{\evhJ}{}^2 \nonumber\\[0.5em]
& \leq \frac{\big(1+\eta^*\big)^2\Cappr^2}{2\varepsilon}\hh^2\snorm{\jv}{1}^2 + \frac{(\eta^*)^2}{2\varepsilon}\norm{\jv-\jvI}{}^2 + \varepsilon\norm{\evhJ}{}^2.
\label{eq:T4:00}
\end{align}
Above, $\jv_\pi$ stands for the piecewise constant best approximant of~$\jv$.

Then, we transform $\norm{\jv-\jvI}{}$ by using the triangle inequality and adding and subtracting~$\PizboldhP\uvI$ and~$\PizboldhP\Bv$:
\begin{align}     \label{eq:T4:10}
\norm{\jv-\jvI}{}
& = \norm{\Ev + \uv\times\Bv - \big( \EvI + \PizboldhP\uvI\times\PizboldhP\Bvh \big)}{} \nonumber\\[0.5em]
& \leq \norm{\Ev-\EvI}{} + \norm{\uv\times\Bv - \PizboldhP\uvI\times\PizboldhP\Bvh}{} \nonumber\\[0.5em]
& \leq \CI\hh\snorm{\Ev}{} + \norm{\big(\uv-\PizboldhP\uvI\big)\times\Bv}{}
      + \norm{\PizboldhP\uvI\times\big(\Bv-\PizboldhP\Bv\big)}{} \nonumber\\[0.5em]
& \hspace{1.26cm}  + \norm{\PizboldhP\uvI\times\big(\PizboldhP\Bv-\PizboldhP\Bvh\big)}{} \nonumber\\[0.5em]
&=: \CI\hh\snorm{\Ev}{} + \TERM{T}{4,1} + \TERM{T}{4,2} + \TERM{T}{4,3}.
\end{align}
  Next, we estimate separately the three terms $\TERM{T}{4,1}$,
  $\TERM{T}{4,2}$, and $\TERM{T}{4,3}$.
We control~$\TERM{T}{4,1}$ by adding and substracting $\PizboldhP\uv$, and using the triangle inequality:
\begin{align*}
\TERM{T}{4,1}
&= \norm{\big(\uv-\PizboldhP\uvI\big)\times\Bv}{} \leq \norm{\Bv}{\LS{\infty}}\norm{\uv-\PizboldhP\uvI}{} \\[0.5em]
&\leq \norm{\Bv}{\LS{\infty}}\Big(\norm{\uv-\PizboldhP\uv}{} + \norm{\PizboldhP\big(\uv-\uvI\big)}{}\Big) 
\leq \norm{\Bv}{\LS{\infty}}\,\hh (\Cappr + \CI) \norm{\uv}{1}.
\end{align*}
We control~$\TERM{T}{4,2}$ by using Lemma~\ref{lemma:chi:error}:
\begin{align*}
\TERM{T}{4,2}
& =    \norm{\PizboldhP\uvI\times\big(\Bv-\PizboldhP\Bv\big)}{}
    \leq \norm{\PizboldhP\uvI}{\LS{\infty}}\,\norm{\Bv-\PizboldhP\Bv}{}\\[0.5em]
& \leq (\Cinv \CI \hE^{\frac12} + \Cinv \CD \CSob) \norm{\uv}{2}\,\hh\Cappr\snorm{\Bv}{1}.
\end{align*}
We control~$\TERM{T}{4,3}$ by adding and subtracting~$\BvII$ defined in~\eqref{eq:BvI:intp} and using again Lemma~\ref{lemma:chi:error}:
\begin{align*}
\TERM{T}{4,3}
&= \norm{\PizboldhP\uvI\times\big(\PizboldhP(\Bv-\Bvh)}{}
    \leq \norm{\PizboldhP\uvI}{\LS{\infty}}\,\norm{\Bv-\Bvh}{} \\[0.5em]
&\leq \norm{\PizboldhP\uvI}{\LS{\infty}}\,\Big(\norm{\Bv-\BvII}{} + \norm{\BvI-\Bvh}{}\Big) \\[0.5em]
&\leq (\Cinv \CI \hE^{\frac12} + \Cinv \CD \CSob) \norm{\uv}{2} (\hh\CI\snorm{\Bv}{1} + \norm{\evhB}{}).
\end{align*}
We collect the three above estimates in~\eqref{eq:T4:10}:
\begin{align*}
& \norm{\jv-\jvI}{} \\[0.5em]
&\leq \hh \Big( \CI \snorm{\Ev}{1} + (\Cappr+\CI) \norm{\Bv}{\LS{\infty}} \norm{\uv}{1}
+ (\Cinv \CI \hE^{\frac12} + \Cinv \CD \CSob) (\Cappr + \CI) \norm{\uv}{2}\snorm{\Bv}{1} \Big) \\[0.5em]
& \qquad  + (\Cinv \CI \hE^{\frac12} + \Cinv \CD \CSob) \norm{\uv}{2}\norm{\evhB}{},
\end{align*}
and this inequality in~\eqref{eq:T4:00}, and obtain the following upper bound for $\TERM{T}{4}$:
\[
\TERM{T}{4}
\leq \hh^2 C_{\TERM{T}{4}}(\varepsilon)
+ \frac{(\eta^*)^2}{\varepsilon}  (\Cinv \CI \hE^{\frac12} + \Cinv \CD \CSob)^2 \norm{\uv}{2}^2 \norm{\evhB}{}^2 
+\varepsilon \norm{\evhJ}{}^2,
\]
where~$C_{\TERM{T}{4}}(\varepsilon) \approx \varepsilon^{-1}$ collects all the constants above and depends also on~$\snorm{\Bv}{}$,
$\snorm{\jv}{1}$, $\snorm{\Ev}{1}$, $\norm{\Bv}{\LS{\infty}}$, and~$\norm{\uv}{2}$.

\medskip
\paragraph*{Estimate of $\TERM{T}{5}$}
We proceed with (by now) standard techniques based on the consistency and stability properties~\eqref{consistency:velocity:H1} and~\eqref{eq:velocity:H1:stability}:
given~$\uv_\pi$ the piecewise linear best approximant of~$\uv$,
\begin{align*}
\TERM{T}{5}
& = \sum_{\E \in \taun} \left[ \asP(\uv,\evhu) - \ashP(\uvI,\evhu)  \right]
  = \sum_{\E \in \taun} \left[ \asP(\uv-\uv_\pi,\evhu) - \ashP(\uvI-\uv_\pi,\evhu)  \right]\\[0.5em]
& \le \sum_{\E \in \taun} \big[ \snorm{\uv-\uv_\pi}{1,\E}
    +\sigma^* \left( \snorm{\uv-\uvI}{1,\E} + \snorm{\uv-\uv_\pi}{1,\E}  \right) \big] \snorm{\evhu}{1,\E} \\[0.5em]
& \le \hh \left( (1+\sigma^*)\Cappr +\sigma^*\CI) \right) \snorm{\uv}{2} \snorm{\evhu}{1}
  \le  \hh^2 \ C_{\TERM{T}{5}}(\varepsilon) + \varepsilon \snorm{\evhu}{1}^2,
\end{align*}
where
\[
C_{\TERM{T}{5}}(\varepsilon)
:= \frac{1}{4\varepsilon}  \left( (1+\sigma^*)\Cappr +\sigma^*\CI) \right)^2 \snorm{\uv}{1}^2.
\]
\medskip
\paragraph*{Estimate of $\TERM{T}{6}$}
Standard manipulations based on the Poincar\'e and the Young inequalities yield
\[
\begin{split}
\TERM{T}{6}
& = (\fv,\Piz1 \evhu) - (\fv, \evhu) 
  = (\fv - \Piz1 \fv, \Piz1\evhu - \evhu)\\
& \le \norm{\fv - \Piz1 \fv}{} \norm{\evhu - \Piz1\evhu}{} 
 \le \Cappr^2 \hh^2 \norm{\fv}{1} \snorm{\evhu}{1}
  \le \hh^2 C_{\TERM{T}{6}} + \varepsilon \snorm{\evhu}{1}^2,
\end{split}
\]
where
\[
C_{\TERM{T}{6}} := \frac14 \hh^2 \Cappr^2 \norm{\fv}{1}^2.
\]

\medskip
\paragraph*{Use of the Gr\" onwall's inequality and final step}
Combining the above bounds on the term~$\TERM{T}{i}$, $i=1,\dots,6$, we get
\begin{align*}
\sum_{i=1}^6\TERM{T}{i}
& \le  S_1(\varepsilon) \hh^2 
  + S_2(\varepsilon) \norm{\evhu}{}^2 
  + 5\varepsilon \snorm{\evhu}{1}^2 
  + S_3(\varepsilon) \norm{\evhB}{}^2 + \varepsilon \norm{\evhJ}{}^2,
\end{align*}
where
\begin{align*}
S_1(\varepsilon) & = \sum_{j=1}^6  C_{\TERM{T}{j}}(\varepsilon),
\qquad S_3(\varepsilon) = \frac12 \norm{\jv}{\LS{\infty}} + \frac{(\eta^*)^2}{\varepsilon} (\Cinv \CI \hE^{\frac12} + \Cinv \CD \CSob)^2 \norm{\uv}{2}^2, \\[0.5em]
S_2(\varepsilon) &= \varepsilon + (1+1\slash\varepsilon) C_1 R_1(\uv) 
+ \norm{\jv}{\LS{\infty}}.
\end{align*}
The terms~$S_j(\varepsilon)$, $j=1,2,3$, depend on the time instant~$t$,
since they depend on~$\uv$, $\Bv$, and~$\jv$.
For the presentation's sake, we do not express this dependence explicitly in our notation.

We substitute the above bound in~\eqref{eq:RHS:T-terms} and the resulting inequality in~\eqref{eq:LHS=RHS}.
Next, we pick
\[
\varepsilon = \varepsilontilde= \min(\alpha_*^{-1}, \eta_*^{-1})\slash 10
\]
and, at every time instant~$t \in (0,T]$, we obtain
\[
\begin{split}
\LHS{} 
& \le 2 \left[ S_1(\varepsilontilde)\hh^2 + S_2(\varepsilontilde) \norm{\evhu}{}^2  + S_3(\varepsilontilde)\, \norm{\evhB}{}^2 \right]\\[0.5em]
& \le 2 \left[ S_1(\varepsilontilde)\hh^2 + \mu_*^{-1} S_2(\varepsilontilde) \msh(\evhu,\evhu)  + \chi_*^{-1} S_3(\varepsilontilde)\, \scalFace{\evhB}{\evhB} \right],
\end{split}
\]
where~$\mu_*$ and~$\chi_*$ are the stability constants defined in~\eqref{eq:velocity:L2:stability} and~\eqref{stability:face}, respectively.

By integrating in time both sides between~$0$ and~$t$, we find that
\begin{align*}
& \norm{\evhu(t)}{\msh}^2 + \norm{\evhB(t)}{\VvhFace}^2
+ \int_0^t \norm{\evhu(s)}{\Wvh}^2
+ \int_0^t \norm{\evhJ(s)}{\VvhEdge}^2 \\[0.5em]
& \leq \norm{\evhu(0)}{\msh}^2 + \norm{\evhB(0)}{\VvhFace}^2 
  + \hh^2\int_0^t \calG_1(s) 
  + \int_0^t \calG_2(s)\Big( \norm{\evhu(s)}{\msh}^2 + \norm{\evhB(s)}{\VvhFace}^2 \Big),
\end{align*}
where 
\begin{align*}
\calG_1(s) = 2 S_1(\varepsilontilde),
\qquad
\calG_2(s) = 2 \mu_*^{-1}S_2(\varepsilontilde) + 2 \chi_*^{-1} S_3(\varepsilontilde).
\end{align*}
The assertion of the theorem follows from an application of Gr\" onwall's inequality, i.e.,
\begin{align*}
& \norm{\evhu(t)}{\msh}^2 + \norm{\evhB(t)}{\VvhFace}^2 
+ \int_0^t \norm{\evhu(s)}{\Wvh}^2 
+ \int_0^t \norm{\evhJ(s)}{\VvhEdge}^2 \\[0.5em]
& \qquad\leq \Cnst{\calG}(t)\big( \norm{\evhu(0)}{\msh}^2 + \norm{\evhB(0)}{\VvhFace}^2  \big) \\[0.5em]
& \qquad\quad + \hh^2 \Cnst{\calG}(t)\int_0^t\calG_1(s)
\quad\textrm{with}\quad
\Cnst{\calG}(t) = \exp\left( \int_0^t\calG_2(s) \right),
\end{align*}
taking the square root on both sides of the above inequality,
and using the stability properties~\eqref{eq:velocity:L2:stability},
\eqref{stability:edge}, and~\eqref{stability:face} on both sides.
\end{proof}

\begin{rem} \label{remark:pressure-convergence}
A linear error bound in the $L^2(0,T;L^2(\Omega))$ norm for the pressure variable could be derived using techniques
based on the discrete inf-sup condition and a suitable error bound on the time derivatives of the velocity and magnetic field.
\end{rem}


\section{Numerical experiments}\label{section:numerical:experiments}

In this section, we validate the theoretical results obtained in Theorem~\ref{thm:convergence:2} and Remark~\ref{remark:pressure-convergence}
with numerical experiments on a manufactured solution problem.
We also show that the discrete velocity and magnetic fields are divergence free up to the conditioning of the final system.
To deal with time derivatives, we used an implicit Euler scheme.
As for the treatment of the nonlinear terms, see~\eqref{eq:VEM:weak:A} and~\eqref{eq:VEM:weak:B},
we employed a fixed-point strategy.
In our experiments, only few fixed-point iterations were required for the convergence of the method at each time step (typically only 3 nonlinear iterations were sufficient).
The resulting linear systems are always solved with a direct
solver~\cite{MUMPS}.
Moreover, to speed up the execution and reduce the computational
time, the assembling and solving of the linear system at hand were run
in parallel.

Given the computational domain~$\Omega:=[0,\,1]^3$, $\REY=\REM=s=1$, and the final time~$T=1$,
we consider the following exact solution:
\begin{eqnarray*}
\Bv &:=& \left[\begin{array}{c}
     4y^3 - 4z^3 - t(24y - 24z)\\
     -3x^2 + 6t\\
     -3y^2 + 6t
\end{array}\right],\\
\Ev &:=& \left[\begin{array}{c}
     6y\\
     12z^2 - 24t\\
     12y^2 - 24t + 6x
\end{array}\right],\\
\uv &:=& \left[\begin{array}{c}
     \sin(\pi x)\cos(\pi y)\cos(\pi z)\cos(t)\\
     \cos(\pi x)\sin(\pi y)\cos(\pi z)\cos(t)\\
     -2\cos(\pi x)\cos(\pi y)\sin(\pi z)\cos(t)
\end{array}\right],\\
p &:=& \left(x^2 + y\,z + z - \frac{13}{12}\right)\cos(t)\,.
\end{eqnarray*}
We compute the right-hand side~$\fv$ in~\eqref{eq:continuous:problem-strong:a} accordingly.
Observe that we are required to add a nonhomogeneous current density on the right-hand side of~\eqref{eq:continuous:problem-strong:b}.

Since the discrete fields~$\Bv_h$, $\Ev_h$, and~$\uv_h$ are not available in closed form,
we measure approximate relative $L^2$-errors involving suitable polynomial projections.
On the other hand, the discrete pressure~$p_h$ is piecewise constant over~$\Th$.
Thence, we compute the following four error quantities:
\begin{equation} \label{computed-errors}
\frac{\norm{\uv-\Pizv{1}\uvh}{}}{\norm{\uv}{}};\qquad
\frac{\norm{\Ev-\Pizv{0}\Evh}{}}{\norm{\Ev}{}};\qquad
\frac{\norm{\Bv-\Pizv{0}\Bvh}{}}{\norm{\Bv}{}};\qquad
\frac{\norm{p-\psh}{}}{\norm{p}{}}.
\end{equation}
For all quantities, the expected convergence rate is linear.

To verify that~$\Bv_h$ and~$\uv_h$ are divergence free up to machine precision,
we compute the~$L^2$ norm of their divergence, i.e., $\div(\Bv_h)$ and $\div(\uv_h)$,
which are computable using only the degrees of freedom.

We consider three families of meshes on~$\Omega$:
\begin{itemize}
\item \texttt{tetra}: Delaunay tetrahedral meshes; see Figure~\ref{fig:meshes}(a);
\item \texttt{cube}: structured meshes consisting of cubes; see Figure~\ref{fig:meshes}(b);
\item \texttt{voro}: Voronoi tessellations optimized by the Lloyd algorithm; see Figure~\ref{fig:meshes}(c).
\end{itemize}

\begin{figure}[htb]
\centering
\begin{tabular}{ccc}
\includegraphics[width=0.30\textwidth]{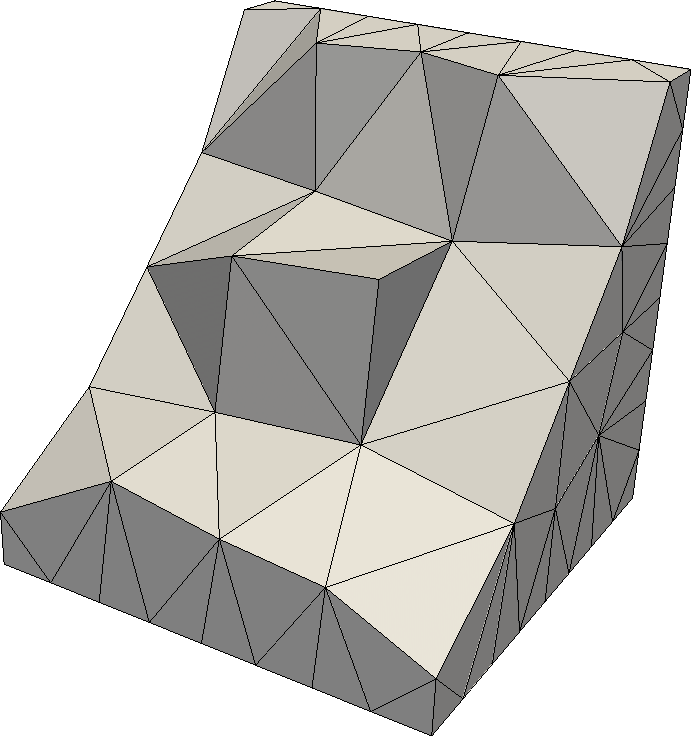} &
\includegraphics[width=0.30\textwidth]{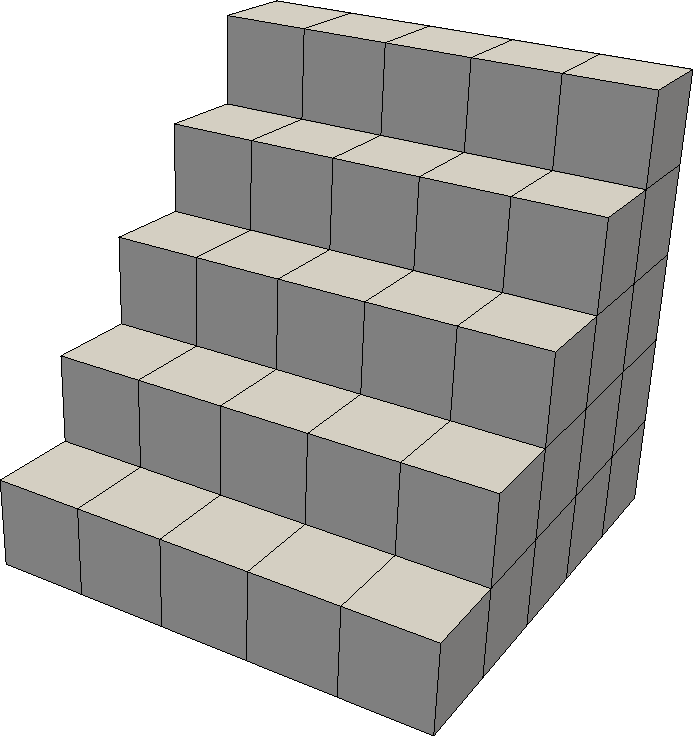} &
\includegraphics[width=0.30\textwidth]{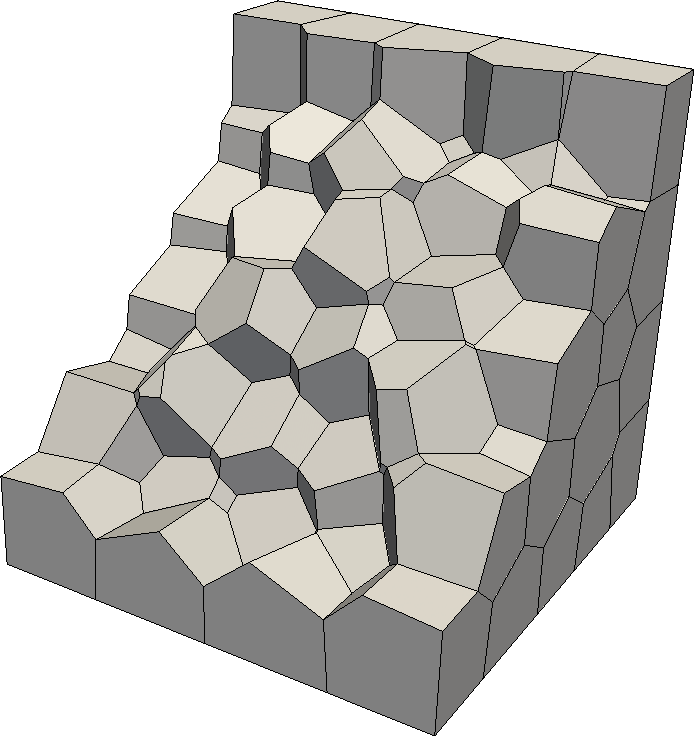} \\
(a)&(b)&(c)
\end{tabular}
\caption{Three representatives of the three family of meshes decomposing the unit cube~$\Omega$.
(a): Delauney mesh; (b) cube mesh; (c) Voronoi-Lloyd mesh.}
\label{fig:meshes}
\end{figure}

To better appreciate the linear trend of all field errors, we refine simultaneously in space and time:
for each mesh family, we consider four meshes with decreasing mesh-size
and use a uniform time discretization of steps~$1/4$, $1/8$, $1/16$, and~$1/32$.

In Figure~\ref{fig:convLines}, we display the errors in~\eqref{computed-errors} for simultaneous space and time refinements.
As expected, we observe linear convergence.
Interestingly, we observe improved convergence by half an order for the velocity field when employing \texttt{cube} and \texttt{voro} meshes, and even a full order improved convergence when employing \texttt{tetra} meshes.
This is in accordance with the fact that the velocity space contains $\Pbb_1$ vector polynomials on each element and we are measuring the $L^2$ norm of the approximation errors.
However that may be, it is an improvement over what is predicted in Theorem~\ref{thm:convergence:2}.

\begin{figure}[htb]
    \centering
    \begin{tabular}{cc}
    \includegraphics[width=0.45\textwidth]{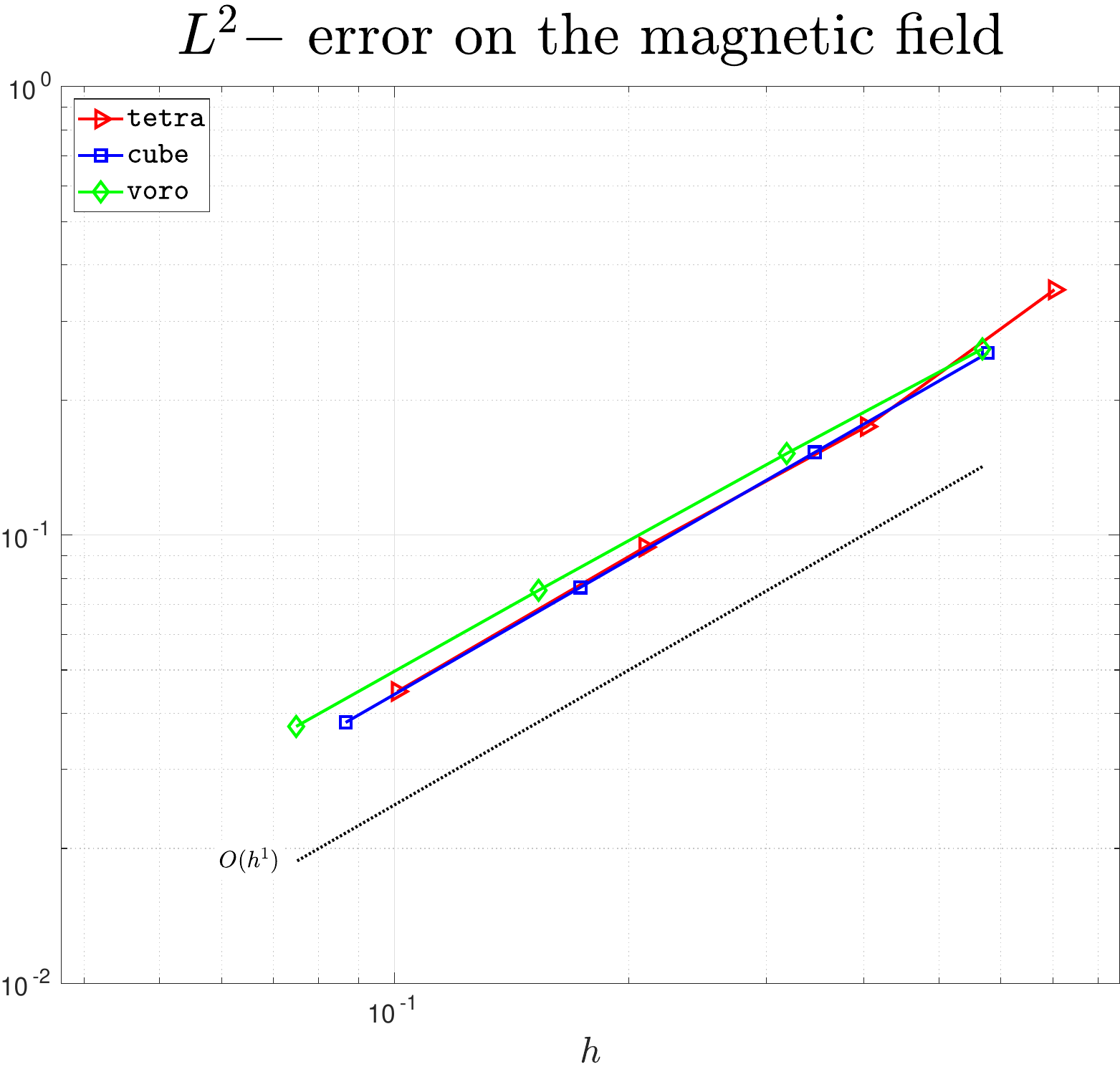} &
    \includegraphics[width=0.45\textwidth]{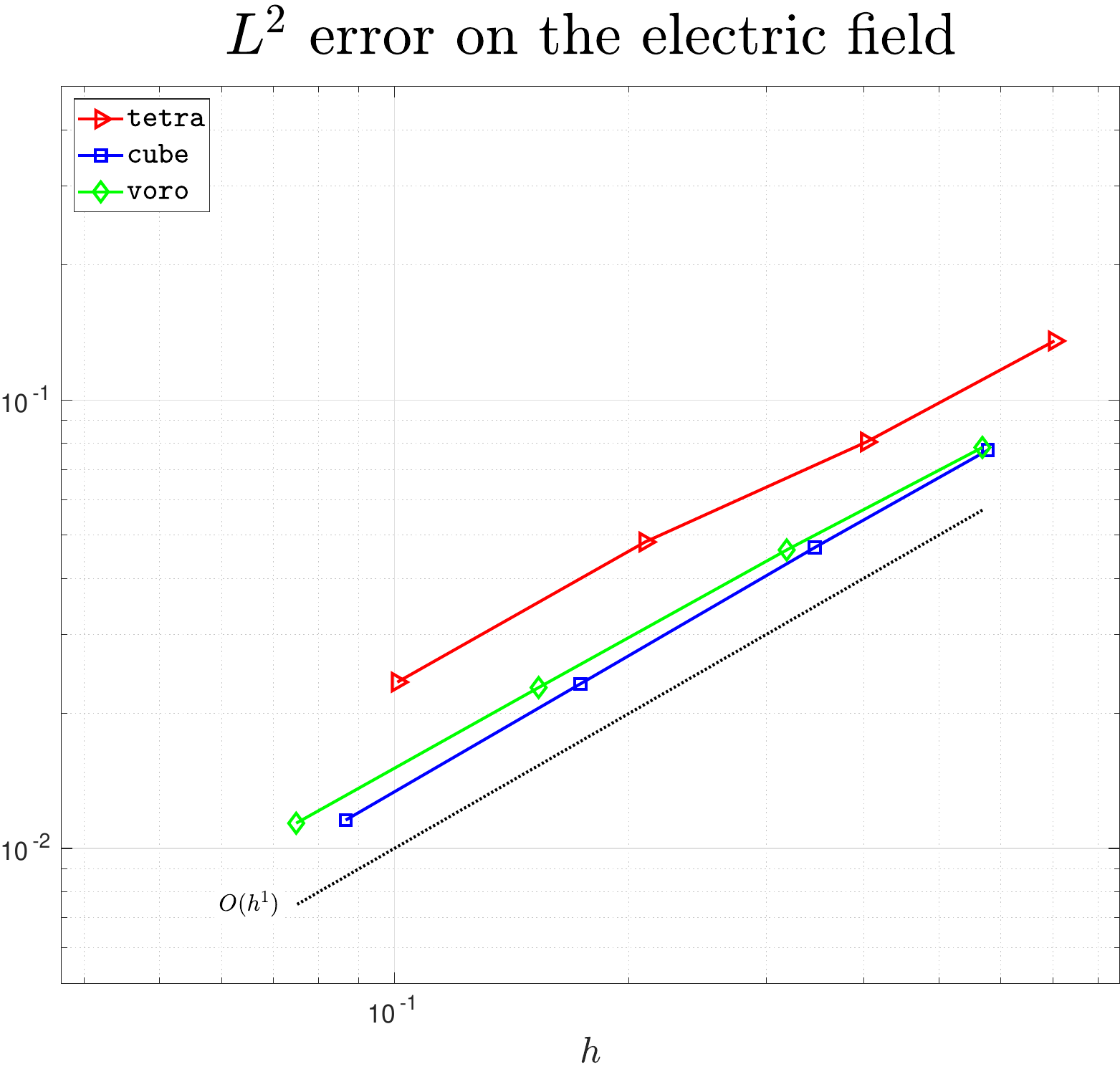} \\
    \includegraphics[width=0.45\textwidth]{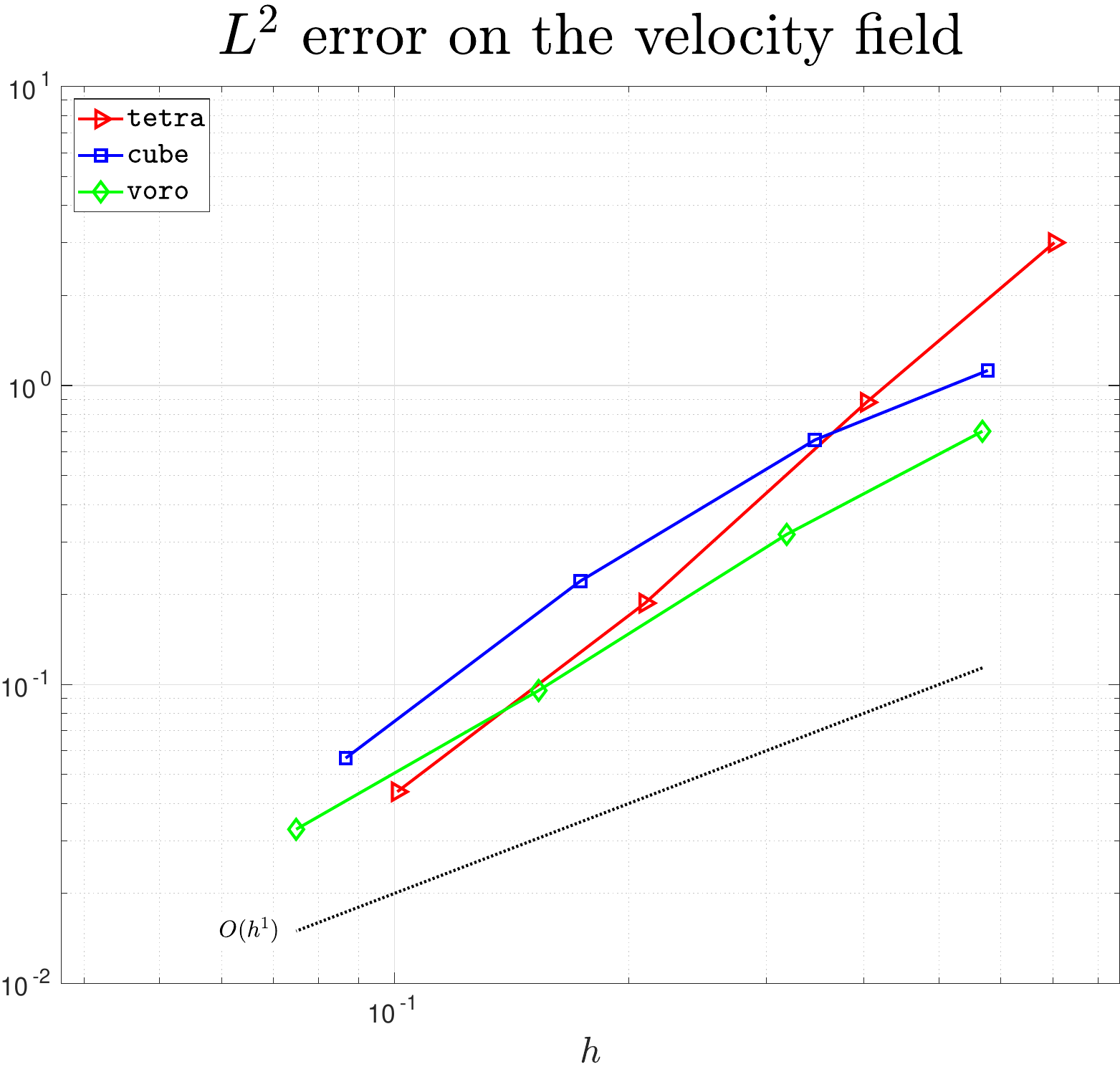} &
    \includegraphics[width=0.45\textwidth]{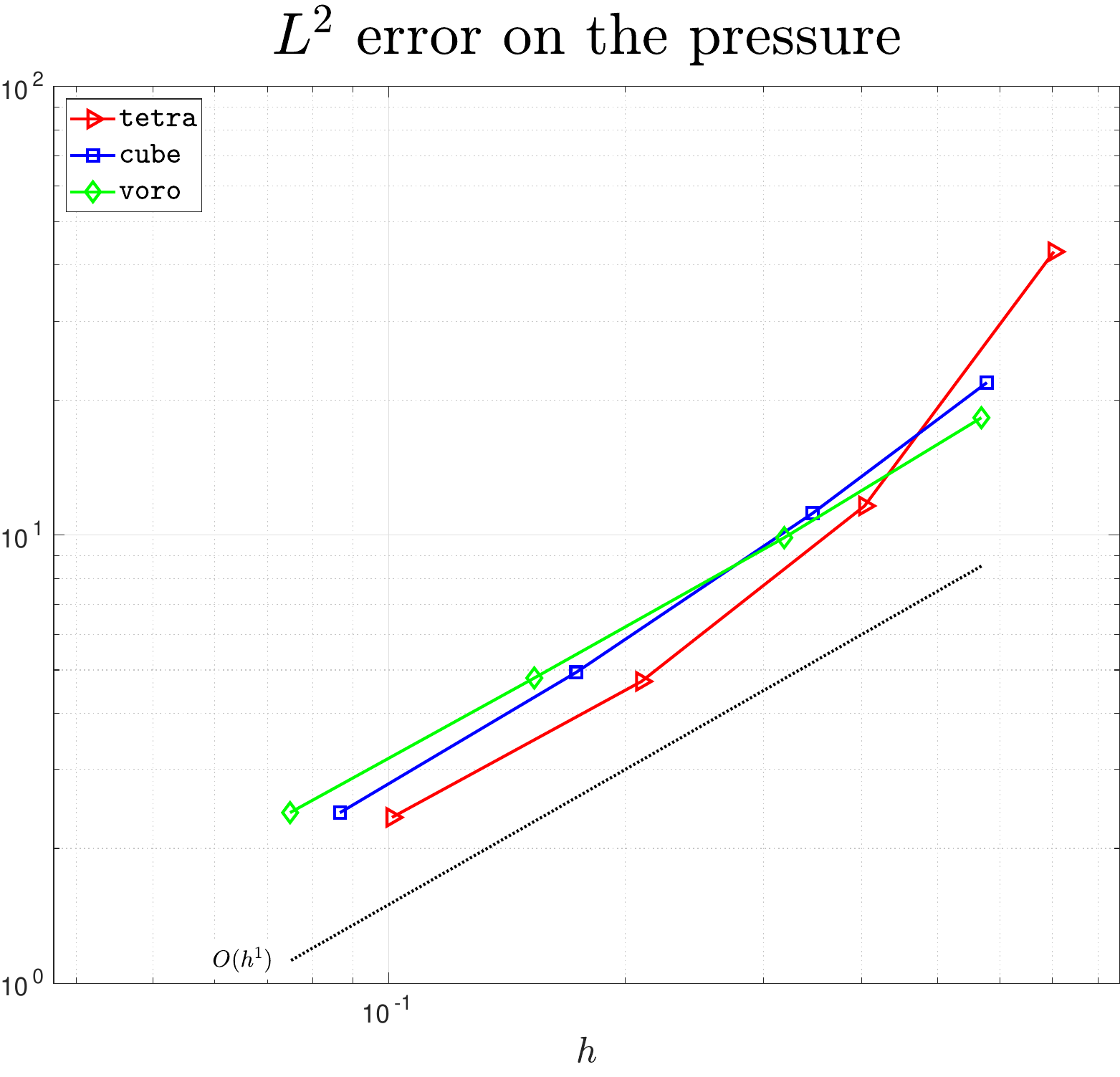} \\
    \end{tabular}
    \caption{Errors~\eqref{computed-errors} for simultaneous mesh refinements,
    using \texttt{tetra}, \texttt{cube}, and \texttt{voro} meshes.}
    \label{fig:convLines}
\end{figure}

In Tables~\ref{tab:divB} and~\ref{tab:divU}, we report the~$L^2$ norm of the divergence of~$\Bvh$ and~$\uvh$, respectively, for each type of meshes and at each space/time refinement step. 
We refer to such refinement levels as~$\texttt{level 1}$, $\texttt{2}$, $\texttt{3}$, and~$\texttt{4}$, where $\texttt{level 1}$ corresponds to the coarsest mesh and the largest time step.
Throughout the refinement process, the value of the norm of the divergence of both fields remains below $1e$-$14$ and $1e$-$11$, respectively, for all choices of space and time refinement.
The growth of such values corresponds to the deterioration of the condition number of the matrices appearing in the linear systems.

\begin{table}[htb]
    \centering
    \begin{tabular}{|r|c|c|c|c|}
    \cline{2-5}
    \multicolumn{1}{c}{}&\multicolumn{4}{|c|}{$\norm{\div \Bvh}{}$}\\
    \cline{2-5}
    \multicolumn{1}{c|}{}&\texttt{level 1}&\texttt{level 2}&\texttt{level 3}&\texttt{level 4}\\
    \hline
    \texttt{tetra} &4.6977e-13 &7.8962e-13 &3.3856e-12 &1.6680e-11\\
    \texttt{cube}  &1.1798e-13 &2.2284e-13 &6.6499e-13 &2.2580e-12\\
    \texttt{voro}  &2.9645e-11 &6.4690e-13 &1.7873e-11 &5.3376e-11\\
    \hline
    \end{tabular}
    \caption{$L^2$ norm of~$\div \Bvh$ for several space and time refinement steps.}
    \label{tab:divB}
\end{table}

\begin{table}[htb]
    \centering
    \begin{tabular}{|r|c|c|c|c|}
    \cline{2-5}
    \multicolumn{1}{c}{}&\multicolumn{4}{|c|}{$\norm{\div \uvh}{}$}
    \\
    \cline{2-5}
    \multicolumn{1}{c|}{}&\texttt{level 1}&\texttt{level 2}&\texttt{level 3}&\texttt{level 4}\\
    \hline
    \texttt{tetra} &7.6676e-16 &2.0355e-15 &1.0726e-14 &6.1851e-14\\
    \texttt{cube}  &1.1714e-15 &1.9226e-15 &7.3605e-15 &4.0470e-14\\
    \texttt{voro}  &2.7855e-16 &3.2315e-15 &1.6780e-14 &8.8101e-14\\
    \hline
    \end{tabular}
    \caption{$L^2$ norm of~$\div \uvh$ for several space and time refinement steps.}
    \label{tab:divU}
\end{table}


\section{Conclusions}
\label{section:conclusions}
We constructed a virtual element method for the approximation of the
solutions to the 3D time-dependent resistive magnetohydrodynamic
model.
The scheme guarantees ``exactly'' divergence-free velocity and
magnetic fields and is suitable to general polyhedral meshes. The
convergence analysis hinges upon the exact sequence structure of the
employed virtual element spaces.
We devoted particular attention to the treatment of the two nonlinear
terms: one is the usual convective term appearing in the Navier-Stokes
equations; the second couples the velocity, magnetic, and electric
fields.
We also validated the theoretical convergence results with a numerical
experiment on a manufactured solution.
Future works will thoroughly cope with the numerical performance of
the method: we shall present more involved and realistic benchmarks
and investigate different approaches to deal with the nonlinear terms.
Further investigations will also cover the design of suitable (direct, iterative, and parallel) solvers to improve the computational efficiency in solving the final linear systems.

\section*{Acknowledgments}
L. Beir\~ao da Veiga was partially supported by the italian PRIN 2017
grant “Virtual Element Methods: Analysis and Applications” and the
PRIN 2020 grant “Advanced polyhedral discretisations of heterogeneous
PDEs for multiphysics problems”.
Both these supports are gratefully acknowledged.
L. Mascotto acknowledges support from the Austrian Science Fund (FWF)
project P33477.
G. Manzini was partially supported by the ERC Project CHANGE, which
has received funding from the European Research Council (ERC) under
the European Union's Horizon 2020 research and innovation program
(grant agreement no.~694515).


{\footnotesize
\bibliography{bibliogr}
\bibliographystyle{plain}
}

\appendix

\end{document}